\documentclass[12pt]{article}
\usepackage{amssymb, amsmath, amsthm, amscd}
\usepackage[dvips]{graphics}
\usepackage[utf8]{inputenc}
\usepackage[all,cmtip]{xy}
\usepackage{bbm}
\usepackage{enumitem}
\usepackage{setspace}
\usepackage[colorlinks=true,
            linkcolor=blue,
            urlcolor=blue,
            citecolor=blue]{hyperref}

\begin{document}

\newtheorem{lem}{Lemma}[section]
\newtheorem{pro}[lem]{Proposition}
\newtheorem{defi}[lem]{Definition}
\newtheorem{def/not}[lem]{Definition/Notations}
\newtheorem{thm}[lem]{Theorem}
\newtheorem{ques}[lem]{Question}
\newtheorem{cor}[lem]{Corollary}
\newtheorem{rem}[lem]{Remark}
\newtheorem{rqe}[lem]{Remarks}
\newtheorem{exa}[lem]{Example}
\newtheorem{exas}[lem]{Examples}
\newtheorem{obs}[lem]{Observation}
\newtheorem{corcor}[lem]{Corollary of the corollary}
\newtheorem*{ackn}{Acknowledgements}

\newcommand{\C}{\mathbb{C}}
\newcommand{\R}{\mathbb{R}}
\newcommand{\N}{\mathbb{N}}
\newcommand{\Z}{\mathbb{Z}}
\newcommand{\Q}{\mathbb{Q}}
\newcommand{\Proj}{\mathbb{P}}
\newcommand{\Rc}{\mathcal{R}}
\newcommand{\Oc}{\mathcal{O}}
\newcommand{\diff}{\textit{diff}}

\begin{center}

{\Large\bf H\"{o}lder continuous solutions to quaternionic Monge-Amp\`{e}re equations}

\end{center}

\begin{center}

{\large Fadoua Boukhari}

\end{center}

%\author{Boukhari Fadoua}

%\address{Boukhari fadoua\\Universite Ibn Tofail\\
%\email{boukhari.fadoua91@gmail.com}
%\date{\today}
%\maketitle
%\setcounter{tocdepth}{1}

\vspace{1ex}

\noindent{\small{\bf Abstract.}
We prove the H\"{o}lder continuity of the unique solution to quaternionic Monge-Amp\`{e}re equation with densities in $L^{p},$ $p>2,$ on a bounded strictly pseudoconvex domains.
\vspace{1ex}

\section*{Introduction}\label{section:introduction}
Recently, people are interested in developing Quaternion analysis, which has become an important branch of mathematics, has many application in mathematical physics. The quaternionic Monge-Ampère operator is defined as the Moore determinant of the quaternionic Hessian of $u$: $$ det(u)=det[\frac{\partial^{2}u}{\partial q_{j}\partial\overline{q}_{k}}(q)].$$
The following Dirichlet problem for the quaternionic Monge-Amp\`{e}re equation in $\Omega\subset\mathbb{H}^{n}:$
\begin{equation}\label{DP11}
\left\{
     \begin{array}{ll}
        det[\frac{\partial^{2}u}{\partial q_{j}\partial\overline{q}_{k}}(q)]=f,\\
        \displaystyle\lim_{q'\rightarrow q }u(q')=\varphi(q), \;\;\;\; \forall q\in \partial\Omega, \varphi\in C(\partial\Omega)
     \end{array}
   \right.
\end{equation}
It has been shown by Alesker \cite[Theorem 1.3]{A3} that (\ref{DP11}) is solvable when $\Omega$ is a strictly pseudoconvex domain and $f\in C(\overline{\Omega}),$ $f\geq0,$ $\varphi\in C(\partial\Omega)$ and the solution is continuous on $\overline{\Omega}.$ For the smooth case, in \cite[Theorem 1.4]{A3} S.Alesker proved a result on existence and uniqueness of the smooth solution of (\ref{DP11}) when the domain $\Omega$ is the Euclidean ball $B$ in $\mathbb{H}^{n}$ and $f\in C^{\infty}(\overline{\Omega}),$ $f>0,$ $\varphi\in C^{\infty}(\partial\Omega).$ He said the reason why he failed to solve
(\ref{DP11}) on general strictly pseudoconvex bounded domains is the fact that the class of diffeomorphisms preserving the class of quaternionic plurisubharmonic (psh) functions must be affine transformations. Relating to this problem, the Dirichlet problem for quaternionic Monge-Amp\'{e}re equations on arbitrary strictly pseudoconvex bounded domains was an open problem. For solving this issue, Zhu proved in \cite{Z} the existence of a subsolution to the Dirichlet problem in quaternionic strictly pseudoconvex bounded domain. By this end and the fact that the subsolutions lead the solutions \cite[Theorem 1.1]{Z} Zhu proved that (\ref{DP11}) is solvable when $\Omega$ is a strictly pseudoconvex bounded domain and $f\in C^{\infty}(\overline{\Omega}),$ $f>0,$ $\varphi\in C^{\infty}(\partial\Omega)$ and the solution is in $C^{\infty}(\overline{\Omega}).$\\
Sroka in \cite{SM} found a continuous solution of this problem (\ref{DP11}) under the much milder assumption $f\in L^{p}(\Omega),$ $p>2.$ \\ To develop the quaternionic pluripotential theory, Alesker defined the quaternionic Monge-Ampère operator on general quaternionic manifolds, he introduced in \cite{A2} an operator in terms of the Baston operator $\Delta,$ which is the first operator of the quaternionic complex on quaternionic manifolds. The n-th power of this operator is exactly the quaternionic Monge-Ampère operator when the manifold is flat. On the flat space $\mathbb{H}^{n},$ the Baston operator $\Delta$ is the first operator of $0$-Cauchy-Fueter complex: \begin{eqnarray}\label{ff}
0\longrightarrow C^{\infty}(\Omega,\mathbb{C})\longrightarrow^{\Delta} C^{\infty}(\Omega,\wedge^{2}\mathbb{C}^{2n})\longrightarrow^{D}C^{\infty}(\Omega,\wedge^{3}\mathbb{C}^{2n})\longrightarrow\ldots
         \end{eqnarray}

Wang \cite{W1} wrote down explicity each operator of the $k$-Cauchy-Fueter complex in terms af real variables.\\
Motivated by this, D.Wan and W.Wang introduced in \cite{WW} two first-order differentiel operators $d_{0}$ and $d_{1}$ acting on the quaternionic version of differentiel forms. The second operator $D$ in (\ref{ff}) can be written as $D:=\left(
                                                                               \begin{array}{c}
                                                                                 d_{0} \\
                                                                                 d_{1} \\
                                                                               \end{array}
                                                                             \right).$
The behavior
of $d_{0},$ $d_{1},$ and $\Delta = d_{0}d_{1}$ is very similar to $\partial,$ $\overline{\partial},$ and $\partial\overline{\partial}$ in several complex variables.
The quaternionic Monge-Ampère operator can be defined as $(\Delta u)^{n}=(d_{0}d_{1}u)^{n}$ and has a simple explicit expression, which is much more convenient than the definition by using Moore determinant. Based on this observation, some authors established and developed the quaternionic versions of several results in complex pluripotentiel theory (for more informations see \cite{WW,WZ,WK}).\\
Motivated by this, we consider The following Dirichlet problem for the quaternionic Monge-Amp\`{e}re equation in a given strictly pseudoconvex domain $\Omega\subset\mathbb{H}^{n}:$
\begin{equation}\label{DP12}
\left\{
     \begin{array}{ll}
        u\in\mathcal{PSH}(\Omega)\cap C(\overline{\Omega}),\\
        (\Delta u)^{n}=d\nu,\\
        \displaystyle\lim_{q'\rightarrow q }u(q')=\psi(q), \;\;\;\; \forall q\in \partial\Omega, \psi\in C(\partial\Omega)
     \end{array}
   \right.
\end{equation}

The purpose of this paper is to study the regularity of solutions to this problem. To begin with, we describe the background. The H\"{o}lder continuous solutions to complex Monge-Amp\`{e}re equations was proved by \cite{GKZ}. In particulier, it is proved that the solution is H\"{o}lder continuous if $d\nu=fdV,$ $ 0\leq f\in L^{p},$ $p>1,$ and $\varphi$ is H\"{o}lder continuous. Then we are going to follow the method of \cite{GKZ} to prove our main result, which is the following Theorem.\\
\textbf{Theorem.}
\textit{Let $\Omega$ be a bounded strongly pseudoconvex domain of $\mathbb{H}^{n}$ with smooth boundary. Assume that $\psi$ is $C^{1,1}$ on $\partial\Omega$ and $0\leq f\in L^{p}(\Omega),$  for some $p>2.$ Then the unique solution $u\in PSH(\Omega)\cap C(\overline{\Omega})$ to the problem \ref{DP12} for $d\nu=fdV$, belongs to $C^{0,\alpha}(\overline{\Omega})$ for any $0<\alpha<\frac{2}{qn+1+\frac{nq}{\frac{2}{q}-1}},$ where $\frac{1}{p}+\frac{1}{q}=1.$}\\

The paper is organized as follows. In section \ref{11111}, we recall basic facts about plurisubharmonic functions, and the quaternionic Monge-Amp\`{e}re operator. In section \ref{mod}, we give an estimate of the modulus of continuity of the solution to the Dirichlet problem for the quaternionic Monge-Amp\`{e}re equation, and prove its useful consequence (Corollary \ref{cor2}) which plays key role in the rest. In section \ref{SE}, we prove our main tool which is the stability estimate. In section \ref{SDPPH}, we show that the unique solution to the quaternionic Monge-Amp\`{e}re equation with densities in $L^{p},$ $p>2,$ is H\"{o}lder continuous if the boundary data $\psi$ is so.\\

\section{Preliminaries}\label{11111}
\subsection{Plurisubharmonic functions of quaternionic variables}
In this part, let us remind few standard notions by \cite{A}.
\begin{defi}
A real valued function $f:\Omega\subset\mathbb{H}^{n}\longrightarrow\mathbb{R}$ is called quaternionic plurisubharmonic if it is upper semi-continous and its restriction to any right quaternionic line is subharmonic.
\end{defi}
\begin{rqe}\label{rqe1}
On $\mathbb{H}^{1}$ the class of plurisubharmonic functions coincides with the class of subharmonic functions in $\mathbb{R}^{4}.$
\end{rqe}
\begin{defi}
Let $\Omega$ be a bounded domain in $\mathbb{H}^{n}.$ Then $\Omega$ is called strictly pseudoconvex if there exists a strictly plurisubharmonic defining function $\varphi,$ i.e, $\Omega=\{q\in\mathbb{H}^{n};\varphi(q)< 0\}.$
\end{defi}
The analogous classical results for subharmonic functions also holds for the quaternionic plurisubharmonic functions. We list these properties here without proofs; all of them can be derived from the subharmonic case (see \cite[Chapter 2]{K}).
\begin{pro}
\begin{enumerate}
  \item If $u\in C^{2}$ then $u$ is plurisubharmonic if and only if the form $\Delta u$ is positive in $\Omega.$
  \item If $u,v\in PSH(\Omega)$ then $\lambda u+\mu v\in PSH(\Omega),$ $\forall\lambda,\mu>0$
  \item If $u$ is plurisubharmonic in $\Omega$ then the standard regularization $u\ast\chi_{\epsilon}$ are also plurisubharmonic in $\Omega_{\epsilon}:=\{q\in\Omega/ d(q,\partial\Omega)>\epsilon\}.$
  \item If $(u_{l})\subset PSH(\Omega)$ is locally uniformly bounded from above then $(\sup u_{l})^{*}\in PSH(\Omega),$ where $v^{*}$ is the upper semi continuous regularization of $v$
  \item $PSH(\Omega)\subset SH(\Omega).$
  \item Let $\emptyset\neq U\subset\Omega$ be a proper open subset such that $\partial U\cap\Omega$ is relativement compact in $\Omega.$ If $u\in PSH(\Omega),$ $v\in PSH(\Omega)$ and $\limsup_{q\longrightarrow q'}v(q)\leq u(q')$ for each $q'\in\partial U\cap\Omega$ then the function $w,$ defined by $$w(q)=\left\{\begin{array}{ll}
                       u(q), & \hbox{$q\in\Omega\setminus U$;} \\
                       \max(u(q),v(q)), &\hbox{$q\in U$.}
                     \end{array}
                   \right.$$
      is plurisubharmonic in $\Omega.$
\end{enumerate}
\end{pro}
Denote by $\mathcal{PSH}$ the class of all quaternionic plurisubharmonic functions (cf.\cite{A,A1,A2}) for more information about plurisubharmonic functions).\\
For the complex case, it is well know that the psh functions are locally integrable with any exponent, but in the quaternionic case we have this following result for local integrability of psh functions.
\begin{pro}(Proposition 2 in \cite{SM})\\
Suppose $u\in PSH(\Omega)$ is such that $u\neq-\infty.$ Then $u\in L^{p}_{loc}(\Omega)$ for any $p<2$ and the bound on $p$ is optimal. What is more if $u_{j}\neq -\infty$ is a sequence of psh functions converging in $L^{1}_{loc}(\Omega)$ to some $u,$ neccessarily belonging to $PSH(\Omega),$ then convergence holds in $L^{p}_{loc}(\Omega)$ for any $p<2.$
\end{pro}
\subsection{The operators $d_{0},$ $d_{1}$ and the Baston operator $\Delta$}
We use the well-known embedding of the quaternionic algebra $\mathbb{H}$ into $End(\mathbb{C}^{2})$ defined by
\begin{eqnarray*}
% \number to remove numbering (before each equation)
 \tau : \mathbb{H} &\longrightarrow& \mathbb{C}^{2\times 2} \\
 \  \ q=x_{0}+ix_{1}+jx_{2}+kx_{3} &\longrightarrow& \left(
                                                   \begin{array}{cc}
                                                     x_{0}-ix_{1} & -x_{2}+ix_{3} \\
                                                     x_{2}+ix_{3} & x_{0}+ix_{1} \\
                                                   \end{array}
                                                 \right)
\end{eqnarray*}
Actually we use the conjugate embedding
\begin{eqnarray*}
      % \number to remove numbering (before each equation)
     \tau : \mathbb{H}^{n}\cong \mathbb{R}^{4n} &\longrightarrow& \mathbb{C}^{2n\times 2}\\
        (q_{0},q_{1},...q_{n-1}) &\longrightarrow& z=(z^{j\alpha})\in \mathbb{C}^{2n\times2}
      \end{eqnarray*}
 with $q_{j}=x_{4j}+ix_{4j+1}+jx_{4j+2}+kx_{4j+3},$ $j=0,1,...,2n-1,$ $\alpha=0,1,$
 with \begin{eqnarray}\label{eq1}
      % \number to remove numbering (before each equation)
           \left(
           \begin{array}{cc}
             z^{00} & z^{01} \\
             z^{10} & z^{11} \\
             \vdots & \vdots \\
             z^{(2l)0} & z^{(2l)1} \\
             z^{(2l+1)0} & z^{(2l+1)1} \\
             \vdots & \vdots \\
             z^{(2n-2)0} & z^{(2n-2)1} \\
             z^{(2n-1)0} & z^{(2n-1)1} \\
           \end{array}
         \right)
         &=& \left(
               \begin{array}{cc}
                 x_{0}-ix_{1} & -x_{2}+ix_{3} \\
                 x_{2}+ix_{3} & x_{0}+ix_{1} \\
                 \vdots & \vdots \\
                 x^{4l}-ix_{4l+1} & -x^{4l+2}+ix_{4l+3} \\
                 x^{4l+2}+ix_{4l+3}& x^{4l}+ix_{4l+1} \\
                 \vdots & \vdots \\
                 x^{4n-4}-ix_{4n-3} & -x^{4n-2}+ix_{4n-1} \\
                 -x^{4n-2}+ix_{4n-1} & x^{4n-4}+ix_{4n-3} \\
               \end{array}
             \right)
      \end{eqnarray}
 Pulling back to the quaternionic space $\mathbb{H}^{n}\cong \mathbb{R}^{4n}$ by the embedding (\ref{eq1}), we define on $\mathbb{R}^{4n}$ first-order differentiel operators $\nabla_{j\alpha}$ as following
 \begin{eqnarray}\label{eq2}
 \left(
    \begin{array}{cc}
      \nabla_{00} & \nabla_{01} \\
      \nabla_{10} & \nabla_{11} \\
       \vdots & \vdots \\
      \nabla_{(2l)0} & \nabla_{(2l)1} \\
      \nabla_{(2l+1)0} & \nabla_{(2l+1)1} \\
      \vdots & \vdots \\
      \nabla_{(2n-2)0} & \nabla_{(2n-2)1} \\
      \nabla_{(2n-1)0} & \nabla_{(2n-1)1} \\
    \end{array}
  \right):
  &=& \left(
             \begin{array}{cc}
               \partial_{x_{0}}+i\partial_{x_{1}} & -\partial_{x_{2}}-i\partial_{x_{3}} \\
               \partial_{x_{2}}-i\partial_{x_{3}} & \partial_{x_{0}}-i\partial_{x_{3}} \\
               \vdots & \vdots \\
               \partial_{x_{4l}}+i\partial_{x_{4l+1}} & -\partial_{x_{4l+2}}-i\partial_{x_{4l+3}} \\
               \partial_{x_{4l+2}}-i\partial_{x_{4l+3}} & \partial_{x_{4l}}-i\partial_{x_{4l+1}} \\
               \vdots & \vdots \\
               \partial_{x_{4n-4}}+i\partial_{x_{4n-3}} & -\partial_{x_{4n-2}}-i\partial_{x_{4n-1}} \\
               \partial_{x_{4n-2}}-i\partial_{x_{4n-1}} & \partial_{x_{4n-4}}-i\partial_{x_{4n-3}} \\
             \end{array}
           \right)
 \end{eqnarray}
  $z^{k\beta}$'s can be viewd as independent variables and $\nabla_{j\alpha}$'s are derivatives with respect to these variables. The operators $\nabla_{j\alpha}$'s play very important roles in the investigating of regular functions in several quaternionic variables.\\
  Let $\wedge^{2k}\mathbb{C}^{2n}$ be the complex exterior algebra generated by $\mathbb{C}^{2n}$, avec $0\leq k\leq n$.\\
  Fixons a basis $\{\omega^{0},\omega^{1},\ldots,\omega^{2n-1}\}$ of $\mathbb{C}^{2n}.$ Let $\Omega$ be a domain in $\mathbb{R}^{4n}.$ we define
   $d_{0},d_{1}:  C^{\infty}_{0}(\Omega,\wedge^{p}\mathbb{C}^{2n})\longrightarrow  C^{\infty}_{0}(\Omega,\wedge^{p+1}\mathbb{C}^{2n})$ by :
   $$d_{0}F=\sum_{k,I}\nabla_{k0}f_{I}\omega^{k}\wedge\omega^{I}$$ $$d_{1}F=\sum_{k,I}\nabla_{k1}f_{I}\omega^{k}\wedge\omega^{I}$$ $$ \triangle F=d_{0}d_{1}F$$
  for  $F=\sum_{I} f_{I}\omega^{I}\in C^{\infty}_{0}(\Omega,\wedge^{p}\mathbb{C}^{2n}),$ where the multi-index $I=(i_{1},\ldots,i_{p})$ et $\omega^{I}=\omega^{i_{1}}\wedge\ldots\wedge\omega^{i_{p}}.$
  The operators $d_{0}, d_{1}$ depend on the choice of coordinates $x_{j}$'s and the basis $\{\omega^{j}\}.$ It is known (cf.\cite{WW}) that the second operator $D$ in the 0-Cauchy-Fueter complex can be written as $DF:=\left(
                                                          \begin{array}{c}
                                                            d_{0}F \\
                                                            d_{1}F \\
                                                          \end{array}
                                                        \right).$\\
   Although $d_{0},d_{1}$  are not exterior differential, their behavior is similar to exterior differential: $d_{0}d_{1}=-d_{1}d_{0}$,  $d_{0}^{2}=d_{1}^{2}=0$; for $F\in C^{\infty}_{0}(\Omega,\wedge^{p}\mathbb{C}^{2n}),$ $G\in C^{\infty}_{0}(\Omega,\wedge^{q}\mathbb{C}^{2n}),$ we have  \begin{eqnarray}\label{eq3}
   d{\alpha}(F\wedge G)=d_{\alpha}F\wedge G+(-1)^{p}F\wedge d_{\alpha}G, \        \ \alpha=0,1, \ \ d_{0}\Delta=d_{1}\Delta=0
   \end{eqnarray}
   We say $F$ is closed if $d_{0}F=d_{1}F=0,$ ie, $DF=0.$ For $u_{1},u_{2},\ldots,u_{n}\in C^{2},$ $\triangle u_{1}\wedge\ldots\wedge\triangle u_{k}$ is closed, $k=1,\ldots,n.$
   Moreover, it follows easily from (\ref{eq3}) that $ \triangle u_{1}\wedge\ldots\wedge\triangle u_{n}$ satisfies the following remarkable identities:
   $$\triangle u_{1}\wedge\ldots\wedge\triangle u_{n}=d_{0}(d_{1}u_{1}\wedge\triangle u_{2}\wedge\ldots\wedge\triangle u_{n})=-d_{1}(d_{0}u_{1}\wedge\triangle u_{2}\wedge\ldots\wedge\triangle u_{n})$$
$$=d_{0}d_{1}(u_{1}\wedge\triangle u_{2}\wedge\ldots\wedge\triangle u_{n})=\triangle(u_{1}\wedge\triangle u_{2}\wedge\ldots\wedge\triangle u_{n}).$$
To write down the explicit expression, we define for a function $u\in C^{2},$
$$\Delta_{ij}u:= \frac{1}{2}(\nabla_{i0}\nabla_{j1}u-\nabla_{i1}\nabla_{j0}u).$$
$2\Delta_{ij}$ is the determinent of $(2\times2)$- submatrix of i-th romws in (\ref{eq2}). Then we can write
$$ \Delta u=\sum_{i,j=0}^{2n-1}\Delta_{ij}u\omega^{i}\wedge\omega^{j},$$
and for $u_{1},...,u_{n}\in C^{2},$
$$\triangle u_{1}\wedge\ldots\wedge\triangle u_{n}=\sum_{i_{1},j_{1},...}\Delta_{i_{1},j_{1}}u_{1}...\Delta_{i_{n},j_{n}}u_{n}\omega^{i_{1}}\wedge\omega^{j_{1}}\wedge...\wedge\omega^{i_{n}}\wedge\omega^{j_{n}}$$
$$ = \sum_{i_{1},j_{1},...}\delta_{01..(2n-1)}^{i_{1}j_{1}..i_{n}j_{n}}\Delta_{i_{1},j_{1}}u_{1}...\Delta_{i_{n},j_{n}}u_{n}\Omega_{2n},$$
where $\Omega_{2n}$ is defined as $$\Omega_{2n}:=\omega^{0}\wedge\omega^{1}\wedge...\wedge\omega^{2n-2}\wedge\omega^{2n-1},$$
and $\delta_{01..(2n-1)}^{i_{1}j_{1}..i_{n}j_{n}}=$ the sign of the permutation from $(i_{1},j_{1}...,i_{n},j_{n})$ to $(0,1,...,2n-1),$ if
$\{i_{1},j_{1}...,i_{n},j_{n}\}=\{0,1,...,2n-1\};$ otherwise,$\delta_{01..(2n-1)}^{i_{1}j_{1}..i_{n}j_{n}}=0.$ Note that $ \triangle u_{1}\wedge\ldots\wedge\triangle u_{n}$ is symmetric with respect to the permutation of $u_{1},...,u_{n}.$ In particulier, when $u_{1}=...=u_{n}=u,$ $ \triangle u_{1}\wedge\ldots\wedge\triangle u_{n}$ coincides with $(\Delta u)^{n}=\wedge^{n}\Delta u.$\\
We denote by $\Delta_{n}( u_{1},..., u_{n})$ the coefficient of the form $ \triangle u_{1}\wedge\ldots\wedge\triangle u_{n},$ ie, $\triangle u_{1}\wedge\ldots\wedge\triangle u_{n}=\Delta_{n}( u_{1},..., u_{n})\Omega_{2n}.$ Then $\Delta_{n}( u_{1},..., u_{n})$ coincides with the mixed Monge-Ampère operator $\det( u_{1},..., u_{n})$ while $\Delta_{n}u$ coincides with the quaternionic Monge-Ampère operator $\det(u),$ we gave an elementary and simpler proof in Appendix A of \cite{WW}.\\
Denote by $\wedge_{\mathbb{R}}^{2k}\mathbb{C}^{2n}$ the subspace of all real elements in $\wedge^{2k}\mathbb{C}^{2n}$ following Alesker \cite{A2}. They are counterparts of $(k,k)-$ forms in several complex variables. In the space $\wedge_{\mathbb{R}}^{2k}\mathbb{C}^{2n}$ Wan and Wang defined convex cones $\wedge_{\mathbb{R}^{+}}^{2k}\mathbb{C}^{2n}$ and $SP^{2k}\mathbb{C}^{2n}$ of positive and strongly positive elements, respectively. Denoted by $\mathcal{D}^{2k}(\Omega)$ the set of all $C_{0}^{\infty}(\Omega)$ functions valued in $\wedge^{2k}\mathbb{C}^{2n}.$ $\eta\in\mathcal{D}^{2k}(\Omega)$ is called a positive form (respectively, strongly positive form) if for any $q\in\Omega,$ $\eta(q)$ is positive (respectively, strongly positive) element. Such forms are the same as the sections of certain line bundle introduced by Alesker \cite{A2} when the manifold is flat. We proved that for $u\in PSH\cap C^{2}(\Omega),$ $\Delta u$ is a closed strongly positive $2$-form.\\
An element of the dual space $(\mathcal{D}^{2n-p}(\Omega))'$ is called a $p$-current. Denoted by $\mathcal{D}_{0}^{p}(\Omega)$ the set of all $C_{0}(\Omega)$ functions valued in $\wedge^{p}\mathbb{C}^{2n}.$ The elements of the dual space $(\mathcal{D}_{0}^{2n-p}(\Omega))'$ are called $p$-currents of order zero. Obviously, the $2n$-currents are just the distributions on $\Omega,$ whereas the $2n$-currents of order zero are Radon measures on $\Omega.$\\
A 2k-current $T$ is said to be positive if we have $T(\eta)\geq0$ for any strongly positive form $\eta\in\mathcal{D}^{2n-2k}(\Omega).$ \\
Although a $2n$-form is not an authentic differentiel form and we cannot integrate it, we can define
$$\int_{\Omega}F:=\int_{\Omega}fdV,$$ if we write $F=f\Omega_{2n}\in L^{1}(\Omega,\wedge^{2n}\mathbb{C}^{2n}),$ where $dV$ is the Lebesgue measure.\\
In particular, if $F$ is positive $2n$-form, then $\int_{\Omega}F\geq0.$ For a $2n$-current $F=\mu\Omega_{2n}$ with coefficient to be measure $\mu,$ define
$$\int_{\Omega}F:=\int_{\Omega}\mu.$$ Any positive $2k$-current $T$ on $\Omega$ has measure coeffucients (i.e.is of order zero)(cf \cite{WW} for more details). For a positive $2k$-current $T$ and a strongly positive test form $\varphi,$ we can write $T\wedge\varphi=\mu\Omega_{2n}$ for some Radon measure $\mu.$ We have
$$ T(\varphi)=\int_{\Omega}T\wedge\varphi.$$
Now for the $p$-current $F,$ we define $d_{\alpha}F$ as $(d_{\alpha}F)(\eta):=-F(d_{\alpha}\eta),$ $\alpha=0,1,$ for any test $(2n-p-1)$-form $\eta.$ We say a current $F$ is closed if $d_{0}F=d_{1}F=0,$ i.e, $DF=0.$ Wan and Wang proved $\Delta u$ is closed positive $2$-current for any $u\in PSH(\Omega).$}
\begin{lem}\label{200}
(Stokes-type formula, \cite[Lemma 3.2]{WW}).\\
Assume that $T$ is a smooth $(2n-1)$-form in $\Omega,$ and $h$ is a smooth function with $h=0$ on $\partial\Omega.$ Then we have
$$\int_{\Omega}hd_{\alpha}T=-\int_{\Omega}d_{\alpha}h\wedge T,\ \  \alpha=0,1,$$
\end{lem}
Bedford-Taylor theory \cite{BT} in complex analysis can be generalized to the quaternionic case. Let $u$ be a locally bounded PSH function and let $T$ be a closed positive $2k$-current. Define $$\Delta u\wedge T:=\Delta(uT),$$ i.e., $(\Delta u\wedge T)(\eta):=uT(\Delta\eta)$ for test form $\eta.$ $\Delta u\wedge T$ is also a closed positive current. Inductively, for $u_{1},\ldots,u_{p}\in PSH\cap L_{loc}^{\infty}(\Omega),$ Wan and Wang showed that
$$\Delta u_{1}\wedge\ldots\wedge\Delta u_{p}:=\Delta(u_{1}\Delta u_{2}\wedge\ldots\wedge\Delta u_{p})$$ is closed positive $2p$-curent. In particular, for $u_{1},\ldots,u_{n}\in PSH\cap L_{loc}^{\infty}(\Omega),$ $\Delta u_{1}\wedge\ldots\wedge\Delta u_{n}=\mu\Omega_{2n}$ for a well-defined positive Radon measure $\mu.$\\
For any test $(2n-2p)$-form $\psi$ on $\Omega,$ we have $$\int_{\Omega}\Delta u_{1}\wedge\ldots\wedge\Delta u_{p}\wedge\psi=\int_{\Omega}u_{1}\Delta u_{2}\wedge\ldots\wedge\Delta u_{p}\wedge\Delta\psi,$$
where $u_{1},\ldots,u_{p}\in PSH\cap L_{loc}^{\infty}(\Omega).$\\
Given a bounded plurisubharmonic function $u$ one can define the quaternionic Monge-Amp\`{e}re measure
$$ (\Delta u)^{n}=\Delta u\wedge\Delta u\wedge\ldots\wedge\Delta u.$$
This is a nonnegative Borel measure.\\
The following capacity was introduced in \cite{WZ} for Borel sets $E\subset\Omega$:
$$ cap(E,\Omega)=\sup\{\int_{E}(\Delta u)^{n}: u\in PSH(\Omega), -1\leq u<0\}.$$
It is closely related to the relative extremal function of the given compact set $K$:
$$ u_{K}(q)=\sup\{u(q): u\in PSH\cap L^{\infty}(\Omega), u<0 \; \hbox{in} \; \Omega, u\leq-1 \; \hbox{on} \; K\},$$
Its upper semicontinuous regalization $u^{*}_{K}(q):= \overline{\lim}_{\zeta\longrightarrow q}u_{K}(\zeta)$ is a plurisubharmonic function and by \cite{WK}we have $$ cap(K,\Omega)=\int_{K}(\Delta u^{*}_{K})^{n}=\int_{\Omega}(\Delta u^{*}_{K})^{n}.$$

W.Wang introduced in \cite{W} the operator $$ \Delta_{a}v:=\frac{1}{2}Re\sum_{j,k=1}^{n} a_{kj}\frac{\partial^{2}v}{\partial \overline{q_{j}}\partial q_{k}}=\frac{1}{2}Re Tr(a(\frac{\partial^{2}v}{\partial \overline{q_{j}}\partial q_{k}}))$$
for $a=(a_{jk})\in\mathcal{H}_{n},$ with $\mathcal{H}_{n}$ the set of all positive quaternionic  hyperhermitian $(n\times n)$ matrices, and a $C^{2}$ real function $v.$ This is an elliptic operator of constant coefficients. This operator is the quaternionic counterpart of complex K\"{a}hler operator, which plays key role in the viscosity approach for the complex case. For more details see \cite{W} and \cite{WW1}.
With the help of this operator, we can prove this following result, by applying the same ideas from the proof of proposition 3.2 in \cite{WW1} and Proposition 5.9 in \cite{GZ17}. \\
We set $$\mathcal{H}'_{n}:=\{a\in\mathcal{H}_{n}/ \det a\geq1\}.$$
\begin{pro}\label{pro1}
Let $u\in PSH(\Omega)\cap L_{loc}^{\infty}(\Omega)$ and $0\leq f\in C(\overline{\Omega}).$ The following conditions are equivalent: \begin{enumerate}
  \item $\Delta_{a}u\geq a_{n}f^{\frac{1}{n}}$ for all $a\in\mathcal{H}'_{n}.$
  \item $(\Delta u)^{n}\geq fdV$ in $\Omega$
\end{enumerate}
where $a_{n}=\frac{n}{2(n!)^{\frac{1}{n}}}.$
\end{pro}
\begin{proof}
$2\Longrightarrow 1.$ Fix $q_{0}\in\Omega$ and $\varphi\in C^{2}$ in neighborhood $B\Subset\Omega$ of $q_{0}$ such that $u\leq\varphi$ in $B$ and $u(q_{0})=\varphi(q_{0}).$ We will prove that $(\Delta \varphi)^{n}_{q_{0}}\geq f(q_{0})dV.$ Suppose by contradiction that $(\Delta \varphi)^{n}_{q_{0}}< f(q_{0})dV,$ by choosing $\epsilon>0$ small enough and letting $\varphi_{\epsilon}:=\varphi+\epsilon|q-q_{0}|^{2},$ we have $0<(\Delta\varphi_{\epsilon})^{n}<fdV$ in $B$ by the continuity of $f.$ It follows from the proof of proposition 3.1 in \cite{WW1} that $\varphi_{\epsilon}$ is plurisubharmonic in $B.$ Now for $\delta>0$ small enough, we have $\varphi_{\epsilon}-\delta\geq u$ near $\partial B$ and $(\Delta \varphi_{\epsilon})^{n}\leq (\Delta u)^{n}.$ The pluripotential comparaison principle yields $\varphi_{\epsilon}-\delta\geq u$ on $B.$ But $\varphi_{\epsilon}(q_{0})=\varphi(q_{0})=u(q_{0}),$ a contradiction. Hence $(\Delta\varphi)^{n}_{q_{0}}\geq f(q_{0})dV.$ Then the hyperhermitian matrix $Q=[\frac{\partial^{2}\varphi}{\partial \overline{q_{j}}\partial q_{k}}(q_{0}))]$ satisfies $\det(Q)\geq f$ at $q_{0}.$ \\ By lemma 3.4 in \cite{WW1}, we have $$(n!\det Q)^{\frac{1}{n}}=(n!)^{\frac{1}{n}}\frac{2}{n}\inf_{a}\Delta_{a}\varphi\geq f^{\frac{1}{n}}$$
for every $a\in\mathcal{H}'_{n}.$ Hence $\Delta_{a}\varphi\geq a_{n}f^{\frac{1}{n}}.$ \\
If $f>0$ is smooth function, there exists $g\in C^{\infty}(\overline{\Omega})$ such that $\Delta_{a}g=a_{n}f^{\frac{1}{n}}.$ Thus $h=u-g$ is subharmonic respect to $\Delta_{a}$ by Proposition 3.2.10 in \cite{H}, and satisfies $\Delta_{a}h\geq0$ in the sense of distributions. Hence $\Delta_{a}u\geq a_{n}f^{\frac{1}{n}}.$\\
If $f>0$ is only continuous, we observe that $$f=\sup\{W, W\in C^{\infty}(\overline{\Omega}), f\geq W>0\}$$
Since $(\Delta u)^{n}\geq fdV,$ we get $(\Delta u)^{n}\geq WdV.$ By the proof above, we can see that $(\Delta_{a}u)\geq a_{n}W^{\frac{1}{n}},$ therefore $\Delta_{a}u\geq a_{n}f^{\frac{1}{n}}.$\\
Now let $f\geq0$ be continuous. We observe that $u_{\epsilon}(q)=u(q)+\epsilon\|q\|^{2}$ satisfies $(\Delta u_{\epsilon})^{n}\geq(f+8^{n}\epsilon^{n})dV,$ since $(\Delta\|q\|^{2})^{n}=8^{n}\beta^{n}_{n}.$ By the last part above, we have $\Delta_{a}u_{\epsilon}\geq a_{n}(f+8^{n}\epsilon^{n})^{\frac{1}{n}}.$ The result follows by letting $\epsilon\longrightarrow0.$\\
$1\Longrightarrow 2.$ Suppose that $u\in C^{2}(\Omega)$ then by lemma 3.4 in \cite{WW1}, we have $\Delta_{a}u\geq a_{n}f^{\frac{1}{n}}$ is equivalent to $(\det(\frac{\partial^{2}u}{\partial \overline{q_{j}}\partial q_{k}}))^{\frac{1}{n}}\geq f^{\frac{1}{n}},$ which it itself equivalent to $(\Delta u)^{n}\geq fdV$ in $\Omega.$\\
If $u$ is not smooth, we consider the standart regularisation $u_{\epsilon}$ of $u$ by convolution with a smoothing kernel. The function $u_{\epsilon}:=u\ast \chi_{\epsilon}$ are plurisubharmonic in $\Omega_{\epsilon}$ and decrease to $u$ as $\epsilon$ decrease to $0.$ We have $\Delta_{a}u_{\epsilon}\geq(a_{n}f^{\frac{1}{n}})_{\epsilon},$ since $u_{\epsilon}$ is smooth, we have $$ (\Delta u_{\epsilon})^{n}\geq ((f^{\frac{1}{n}})_{\epsilon})^{n}dV.$$ Letting $\epsilon\longrightarrow0,$ and applying the convergence theorem for the quaternionic Monge-Amp\`{e}re operator, we get $(\Delta u)^{n}\geq fdV$ in $\Omega.$
\end{proof}
Consider $$\mathcal{U}=\mathcal{U}(\Omega,\psi,f)=\{u\in PSH\cap C(\overline{\Omega}), u_{/\partial\Omega}\leq\psi \ \ and \ \ \Delta_{a}u\geq a_{n}f^{\frac{1}{n}}, \forall a\in \mathcal{H}'_{n}\}$$
It is easy to show that $\mathcal{U}$ is non empty. Then by proposition \ref{pro1}, we can describe the solution as the following
$$ U=\sup\{u\in\mathcal{U}(\Omega,\psi,f)\}.$$
\section{The Modulus of continuity of The solution}\label{mod}
With the help of \cite{C}. we can use in this part the modulus of continuity of the solution to Dirichlet problem for quaternionic Monge-Amp\`{e}re equation (\ref{DP12}).\\
Recall that a real function $\theta$ on $[0,r],$ $0<r<\infty$ is called a modulus of continuity if $\theta$ is continuous, subadditive, nondecreasing and $\theta(0)=0.$ In general, $\theta$ fails to be concave, we denote $\overline{\theta}$ to be the minimal concave majorant of $\theta.$ We denote $\theta_{\varphi}$ the optimal modulus of continuity of the continuous function $\varphi$ which is defined by $$ \theta_{\varphi}(t)=\sup_{|x-y|\leq t}|\varphi(x)-\varphi(y)|.$$
Now, we will prove the following result which is the one of the useful properties of $\overline{\theta}$
\begin{lem}\label{lem1}
Let $\theta$ be a modulus of continuity on $[0,r]$ and $\overline{\theta}$ be the minimal concave majorant of $\theta.$ Then
$\theta(\lambda t)<\overline{\theta}(\lambda t)<(1+\lambda)\theta(t)$ for any $t>0$ and $\lambda>0.$
\end{lem}
\begin{proof}
The same proof of Lemma 3.1 in \cite{C}.
\end{proof}
In the following result, we establish a barrier to the problem (\ref{DP12}) and give an estimate of its modulus of continuity, which will be used in the proof of Theorem \ref{thm1}
\begin{pro}\label{pro2}
Let $\Omega\subset\mathbb{H}^{n}$ be a bounded strongly pseudoconvex domain with smooth boundary, assume that $\theta_{\psi}$ is the modulus of continuity of $\psi\in C(\partial\Omega)$ and $0\leq f\in C(\overline{\Omega}).$ Then there exists a subsolution $u\in\mathcal{U}(\Omega,\psi,f)$ such that $u=\psi$ on $\partial\Omega$ and the modulus of continuity of $u$ satisfies the following inequality $$\theta_{u}(t)\leq\eta\max\{\theta_{\psi}(t^{\frac{1}{2}}),t^{\frac{1}{2}}\},$$ where $\eta=\lambda(1+a_{n}\|f\|^{\frac{1}{n}}_{L^{\infty}(\overline{\Omega})})$ and $\lambda\geq1$ is a constant depending on $\Omega.$
\end{pro}
\begin{proof}
Fix $\xi\in\partial\Omega.$ We will prove that there exists $u_{\xi}\in\mathcal{U}(\Omega,\psi,f)$ such that $u_{\xi}(\xi)=\psi(\xi).$ \\ As in the proof of proposition 3.2 in \cite{C}, and by using Lemma \ref{lem1} we prove that there exists a constant $C>0$ depending only on $\Omega$ such that every point $\xi\in\partial\Omega$ and $\psi\in C(\partial\Omega),$ there is a function $v_{\xi}\in PSH(\Omega)\cap C(\overline{\Omega})$ such that
\begin{enumerate}
  \item $v_{\xi}(q)\leq\psi(q)$ $\forall q\in\partial\Omega$
  \item $v_{\xi}(\xi)=\psi(\xi)$
  \item $\theta_{v_{\xi}}(t)\leq C\theta_{\psi}(t^{\frac{1}{2}}).$
\end{enumerate}
Fix a point $q_{0}\in\Omega$ and choose $K_{1}\geq0$ such that $K_{1}=a_{n}\sup_{\overline{\Omega}}f^{\frac{1}{n}}.$ Then $$\Delta_{a}(K_{1}|q-q_{0}|^{2})=K_{1}\Delta_{a}|q-q_{0}|^{2}\geq a_{n}f^{\frac{1}{n}}(q),$$ for all $a\in\mathcal{H}'_{n}.$ Set $K_{2}=K_{1}|\xi-q_{0}|^{2}.$ Then for the continuous function
$\widetilde{\psi}(q):= \psi(q)-K_{1}|q-q_{0}|^{2}+K_{2}$ we have $v=v_{\xi}$ such that 1,2 and 3 hold.
Then $u_{\xi}\in\mathcal{U}(\Omega,\psi,f)$ is given by $$u_{\xi}(q):= v(q)+K_{1}|q-q_{0}|^{2}-K_{2}$$
Indeed, $u_{\xi}\in PSH(\Omega)\cap C(\overline{\Omega})$ and we have
$v(q)\leq\widetilde{\psi}(q)=\psi(q)-K_{1}|q-q_{0}|^{2}+ K_{2}$ on $\partial\Omega.$ \\ So that $u_{\xi}(q)\leq\psi(q)$ on $\partial\Omega$ and $u_{\xi}(\xi)=\psi(\xi).$ We have
$$\Delta_{a}u_{\xi}=\Delta_{a}v+K_{1}\Delta_{a}|q-q_{0}|^{2}\geq a_{n}f^{\frac{1}{n}} \ \ in \ \ \Omega.$$
Then, by the hypothesis, we can get an estimate for the modulus of continuity of $u_{\xi}$
\begin{eqnarray*}
  \theta_{u_{\xi}}(t)=\sup_{|q-q'|\leq t}|u(q)-u(q')|&\leq& \theta_{v}(t)+K_{1}\theta_{|q-q_{0}|^{2}}(t)\\
   &\leq& C\theta_{\widetilde{\psi}}(t^{\frac{1}{2}})+4d^{\frac{3}{2}}K_{1}t^{\frac{1}{2}}\\
   &\leq& C\theta_{\psi}(t^{\frac{1}{2}})+2dK_{1}(C+2d^{\frac{1}{2}})t^{\frac{1}{2}}\\
   &\leq& (C+2d^{\frac{1}{2}})(1+2dK_{1})\max\{\theta_{\psi}(t^{\frac{1}{2}}),t^{\frac{1}{2}}\}
\end{eqnarray*}
Then, we choose $\lambda$ so that $\theta_{u_{\xi}}(t)\leq\lambda(1+a_{n}\|f\|^{\frac{1}{n}}_{L^{\infty}(\overline{\Omega})})\max\{\theta_{\psi}(t^{\frac{1}{2}}),t^{\frac{1}{2}}\}.$ Hence the desired result follows.
\end{proof}
\begin{cor}\label{cor1}
Taking the same assumption of Proposition \ref{pro2}. There exists a plurisuperharmonic function $\widetilde{u}\in C(\overline{\Omega})$ such that $\widetilde{u}=\psi$ on $\partial\Omega$ and
$$\theta_{\widetilde{u}}(t)\leq\eta\max\{\theta_{\psi}(t^{\frac{1}{2}}),t^{\frac{1}{2}}\},$$ where
$\eta=\lambda(1+a_{n}\|f\|^{\frac{1}{n}}_{L^{\infty}(\overline{\Omega})})$ and $\lambda\geq1$ is a constant depending on $\Omega.$
\end{cor}
\begin{proof}
We can use the same construction as in the proof of Proposition \ref{pro2} for $\psi_{1}=-\psi\in C(\partial\Omega),$ then there exists $u_{1}\in\mathcal{U}(\Omega,\psi_{1},f)$ such that $u_{1}=\psi_{1}$ on $\partial\Omega$ and
$\theta_{u_{1}}(t)\leq\eta\max\{\theta_{\psi_{1}}(t^{\frac{1}{2}}),t^{\frac{1}{2}}\}.$ Then, we set $\widetilde{u}=-u_{1}$ which is plurisuperharmonic function on $\Omega,$ continuous on $\overline{\Omega}$ and satisfies $\widetilde{u}=\psi$ on $\partial\Omega$ and $\theta_{\widetilde{u}}(t)\leq\eta\max\{\theta_{\psi}(t^{\frac{1}{2}}),t^{\frac{1}{2}}\}.$
\end{proof}
Now, we are in position to prove an estimate for the modulus of continuity of the solution to Dirichlet problem for quaternionic Monge-Amp\`{e}re equation.
\begin{thm}\label{thm1}
Let $\Omega$ be a smoothly bounded strongly pseudoconvex domain in $\mathbb{H}^{n},$ suppose that $0\leq f\in C(\overline{\Omega})$ and $\psi\in C(\partial\Omega).$ Then the modulus of continiuity $\theta_{u}$ of the solution $u$ satisfies the following estimate $$ \theta_{u}(t)\leq \gamma(1+a_{n}\|f\|^{\frac{1}{n}}_{L^{\infty}(\overline{\Omega})})\max\{\theta_{\psi}(t^{\frac{1}{2}}),a_{n}\theta_{f^{\frac{1}{n}}}(t),t^{\frac{1}{2}}\}$$
where $\gamma\geq1$ is a constant depending only on $\Omega.$
\end{thm}
\begin{proof}
Thanks to Proposition \ref{pro2}, Corollary \ref{cor1} and the comparaison principle (Corollary 1.1 in \cite{WZ}), we can follow the same proof of Theorem 1.1 in \cite{C}, with setting $g(t)=\max\{\eta\max(\theta_{\psi}(t^{\frac{1}{2}}),t^{\frac{1}{2}}),a_{n}\theta_{f^{\frac{1}{n}}}(t)\}$ and we get the desired result.
\end{proof}
Now, it is easy to check that this previous Theorem has the following consequence.
\begin{cor}\label{cor2}
Let $\Omega$ be a smoothly bounded strongly pseudoconvex domain in $\mathbb{H}^{n}.$ Let $\psi\in Lip_{2\alpha}(\partial\Omega)$ and $0\leq f^{\frac{1}{n}}\in Lip_{\alpha}(\overline{\Omega}),$ $0<\alpha\leq\frac{1}{2}.$ Then the unique solution of Dirichlet problem $u$ is $\alpha$-H\"{o}lder continuous on $\overline{\Omega}.$
\end{cor}
\section{The stability estimate}\label{SE}
In this section, the main goal is to prove the stability estimate, Theorem \ref{4}. For this end, we need some results which are the following:
\begin{lem}\label{1}
Let $u,v\in PSH\cap L^{\infty}(\Omega)$ such that $\underline{\lim}_{\zeta\longrightarrow\partial\Omega}(u-v)(\zeta)>0.$ Then for all $t,s>0,$
$$ s^{n}cap(\{u-v<-t-s\})\leq\int_{\{u-v<-t\}}(\Delta u)^{n}.$$
\end{lem}
\begin{proof}
Take $-1\leq\varphi\leq0$ a psh function in $\Omega.$ We have $\{u-v<-t-s\}\subset\{u<v-t+sv\}\subset\{u<v-t\}\Subset\Omega.$ By the comparaison principle \cite[Theorem 1.2]{WZ} we find
$$ s^{n}\int_{\{u-v<-t-s\}}(\Delta\varphi)^{n}\leq\int_{\{u<v-t+s\varphi\}}(\Delta(-t+v+s\varphi))^{n}\leq\int_{\{u-v<-t\}}( \Delta u)^{n}.$$
Taking the supremum and the lemma follows.
\end{proof}
Now, we are going to prove the following estimate which play an important role in the rest.
\begin{lem}\label{2}
Assume that $0\leq f\in L^{p}(\Omega),$ $p>2,$ and a fixed $\alpha\in (1,2).$ Then there exists a constant $D=D(\alpha,\|f\|_{L^{p}})>0$ such that for every $E\Subset\Omega$
$$ 0\leq\int_{E}fdV\leq D[cap(E)]^{\frac{\alpha}{q}},$$
where $\frac{1}{p}+\frac{1}{q}=1.$
\end{lem}
\begin{proof}
By Holder inequality and using Lemma 3 in \cite{SM}, we have
\begin{eqnarray*}
  \int_{E}fdV &\leq& \|f\|_{L^{p}}V(E)^{\frac{1}{q}}\leq C(\alpha)\|f\|_{L^{p}(\Omega)}[cap(E)]^{\frac{\alpha}{q}}\\
    &\leq& D(\alpha,\|f\|_{L^{p}})[cap(E)]^{\frac{\alpha}{q}}
\end{eqnarray*}
where $\frac{1}{p}+\frac{1}{q}=1,$ hence the lemma follows.
\end{proof}
We will also need the following result, which its proof is similar to Lemma 2.4 in \cite{EGZ}.
\begin{lem}\label{2.4}
Let $f:\mathbb{R}^{+}\longrightarrow\mathbb{R}^{+}$ be a decreasing right-continuous function such that $\lim_{+\infty}f=0.$ Assume there exists $\tau>1,$ $B>0$ such that $f$ satisfies $$ tf(s+t)\leq B[f(s)]^{\tau}, \forall t,s>0.$$ Then, there exists $S_{\infty}:=\frac{2Bf(0)^{\tau-1}}{1-2^{1-\tau}}$ such that $f(s)=0$ for all $s\geq S_{\infty}.$
\end{lem}
\begin{pro}\label{3}
Let $u,v\in PSH\cap L^{\infty}(\Omega)$ be such that $\underline{\lim}_{\zeta\longrightarrow\partial\Omega}(u-v)(\zeta)\geq0$ and $0\leq f\in L^{p}(\Omega),$ $p>2.$ Suppose that $(\Delta u)^{n}=fdV,$ then for any $0<\beta<\frac{1}{n}(\frac{2}{q}-1),$ $\frac{1}{p}+\frac{1}{q}=1,$ there exists a constant $C=C(\alpha,\|f\|_{L^{p}(\Omega)})$ such that for all $\epsilon>0$
$$\sup_{\Omega}(v-u)\leq \epsilon+ C[cap(\{u-v<-\epsilon\})]^{\beta}.$$
\end{pro}
\begin{proof}
By Lemmas \ref{1} et \ref{2}, the function
$ g(s):=[cap(\{u-v<-\epsilon-s\})]^{\frac{1}{n}}$ satisfies the conditions of Lemma \ref{2.4}, we obtain $cap(\{u-v<-s_{\infty}-\epsilon\})=0$ which means that $v-u\leq\epsilon+s_{\infty}$ almost everywhere on $\Omega.$ Finally, if we choose $\tau:=1+\beta n$ we obtain $\sup(v-u)\leq\epsilon+C[cap(\{u-v<-\epsilon\})]^{\beta}$ where \\ $C:=2B/(1-2^{-\beta n}).$
\end{proof}
We are now in the position to prove the main stability estimate, which is similar to Theorem 1.1 in \cite{GKZ} for the complex case.
\begin{thm}\label{4}
Let $u_{1},u_{2}\in PSH\cap L^{\infty}(\Omega)$ be such that $u_{1}\geq u_{2}$ on $\partial\Omega,$ and $0\leq f\in L^{p}(\Omega),$ $p>2.$ Suppose that $(\Delta u_{1})^{n}=fdV$ in $\Omega.$ Fix $r\geq1$ and $0<\gamma<\gamma_{r},$ with $\gamma_{r}=\frac{r}{nq+r+\frac{nq}{\frac{2}{q}-1}},$ $\frac{1}{p}+\frac{1}{q}=1.$ Then there exists a constant $C=C(\gamma,\|f\|_{L^{p}(\Omega)})>0$ such that $$ \sup_{\Omega}(u_{2}-u_{1})\leq C[\|(u_{2}-u_{1})_{+}\|_{L^{r}(\Omega)}]^{\gamma},$$ where $(u_{2}-u_{1})_{+}:=\max(u_{2}-u_{1},0).$
\end{thm}
\begin{proof}
Using Lemma \ref{1} with $s=t=\epsilon>0$ and by H\"{o}lder inequality, we obtain
\begin{eqnarray*}
  Cap(\{u_{1}-u_{2}<-2\epsilon\}) &\leq& \epsilon^{-n}\int_{\{u_{1}-u_{2}<-\epsilon\}}fdV \\
    &\leq& \epsilon^{-n-\frac{r}{q}}\int_{\Omega}(u_{2}-u_{1})_{+}^{\frac{r}{q}}fdV \\
    &\leq& \epsilon^{-n-\frac{r}{q}}\|(u_{2}-u_{1})_{+}\|_{L^{r}(\Omega)}^{\frac{r}{q}}\|f\|_{L^{p}(\Omega)}
\end{eqnarray*}
By Proposition \ref{3}, we get
$$ \sup_{\Omega}(u_{2}-u_{1})\leq 2\epsilon+C\epsilon^{-\beta(n+\frac{r}{q})}\|(u_{2}-u_{1})_{+}\|_{L^{r}(\Omega)}^{\beta\frac{r}{q}}\|f\|_{L^{p}(\Omega)}^{\beta}.$$
Fix $\gamma$ and set $\epsilon:=\|(u_{2}-u_{1})_{+}\|^{\gamma}_{L^{r}(\Omega)},$ we get
$$ \sup_{\Omega}(u_{2}-u_{1})\leq 2\|(u_{2}-u_{1})_{+}\|^{\gamma}_{L^{r}(\Omega)}+C\|(u_{2}-u_{1})_{+}\|_{L^{r}(\Omega)}^{-\gamma\beta(n+\frac{r}{q})+\beta\frac{r}{q}}\|f\|_{L^{p}(\Omega)}^{\beta}.$$
If we choose $\beta=\frac{\gamma q}{r-\gamma(r+nq)},$ we easily obtain the estimate of this Theorem.
\end{proof}

\section{H\"{o}lder continuous solutions to quaternionic Monge-Amp\`{e}re equations}\label{SDPPH}

For a fixed $\delta>0,$ we set $\Omega_{\delta}:=\{q\in\Omega/ dist(q,\partial\Omega)>\delta\};$\\
 $$ u_{\delta}(q):=\sup_{\|\zeta\|<\delta} u(q+\zeta),\ \ q\in\Omega_{\delta};$$
 and  $$ \widehat{u}_{\delta}(q):=\frac{1}{\tau_{4n}\delta^{4n}}\int_{|\zeta-q|\leq\delta}u(\zeta)dV_{4n}(\zeta), \ \ q\in\Omega_{\delta},$$
 where $\tau_{4n}$ is the volume of the unit ball in $\mathbb{H}^{n}.$\\ In the following result, we show the link between $u_{\delta}$ and $\widehat{u}_{\delta}.$
\begin{lem}\label{5}
Given $0<\beta<1,$ the following two conditions are equivalent:
\begin{enumerate}
  \item There exist $\eta_{1},A_{1}>0$ such that for any $0<\delta\leq\eta_{1}$ $$u_{\delta}-u\leq A_{1}\delta^{\beta}, \ \ on \ \ \Omega_{\delta}.$$
  \item There exist $\eta_{2},A_{2}>0$ such that for any $0<\delta\leq\eta_{2}$ $$\widehat{u}_{\delta}-u\leq A_{2}\delta^{\beta}, \ \ on \ \ \Omega_{\delta}.$$
\end{enumerate}
\end{lem}
\begin{proof}
This result is proved in \cite{GKZ} for the complex case, we will follow the same proof of Lemma 4.2 in \cite{GKZ}.
\end{proof}
The content of our next result (Lemma \ref{6}) is to control the growth of $\|u_{\delta}-u\|_{L^{2}(\Omega_{\delta})}$ and $\|\widehat{u}_{\delta}-u\|_{L^{1}(\Omega_{\delta})},$ but before we are in need of this following lemma.
\begin{lem}\label{5.1}
Suppose that $\Omega$ is a domain, $a\in\Omega,$ $B(a,r)\Subset\Omega,$ and $u$ is a psh function. Then for $r>0,$ $q\in\mathbb{H}^{n},$
$$\int_{B(a,r)}\Delta u\wedge\Delta(\frac{-1}{\|q-a\|^{2}})^{n-1}=\frac{1}{r^{4n-4}}\int_{B(a,r)}\Delta u\wedge\beta_{n}^{n-1}.$$
\end{lem}
\begin{proof}
First, we are going to prove that $$\int_{\{a\}}\Delta u\wedge(\Delta(\frac{-1}{\|q-a\|^{2}})^{n-1}=0.$$
It follows from the proof of Proposition 4.1 in \cite{WW}, and by lemma 4.1 in \cite{WW} for $\frac{-1}{\|q-a\|^{2}+\epsilon},$ we get
$$\Delta u\wedge\Delta(\frac{-1}{\|q-a\|^{2}})^{n-1}=\sum_{i_{1}j_{1}\ldots i_{n}j_{n}}\delta_{01\ldots(2n-1)}^{i_{1}j_{1}\ldots i_{n}j_{n}}\Delta_{i_{1}j_{1}}u\Delta_{i_{2}j_{2}}(\frac{-1}{\|q-a\|^{2}+\epsilon})\ldots\Delta_{i_{n}j_{n}}(\frac{-1}{\|q-a\|^{2}+\epsilon}) \Omega_{2n}$$
$$=(\frac{-4}{(\|q-a\|^{2}+\epsilon)^{3}})^{n-1}\sum_{i_{1}j_{1}\ldots i_{n}j_{n}}\delta_{01\ldots(2n-1)}^{i_{1}j_{1} \ldots i_{n}j_{n}}\Delta_{i_{1}j_{1}}u(\overline{M}_{i_{2}j_{2}}-\sum_{k_{2}}\delta_{(2k_{2})(2k_{2}+1)}^{i_{2}j_{2}}(\|q-a\|^{2}+\epsilon))\ldots$$ \\ $\ldots(\overline{M}_{i_{n}j_{n}}-\sum_{k_{n}}\delta_{(2k_{n})(2k_{n}+1)}^{i_{n}j_{n}}(\|q-a\|^{2}+\epsilon))\Omega_{2n}$
$$=(\frac{-4}{(\|q-a\|^{2}+\epsilon)^{3}})^{n-1}[\sum_{k_{1}\ldots k_{n}}2^{n}\delta_{01\ldots(2n-1)}^{(2k_{1})(2k_{1}+1)\ldots(2k_{n})(2k_{n}+1)}\Delta_{(2k_{1})(2k_{1}+1)}u.(-\|q-a\|^{2}-\epsilon)^{n-1}+ $$
$$ \sum_{i_{1}j_{1}i_{2}j_{2}k_{3}\ldots k_{n}}2^{n-2}\delta_{01\ldots(2n-1)}^{i_{1}j_{1} i_{2}j_{2}(2k_{3})(2k_{3}+1)\ldots(2k_{n})(2k_{n}+1)}\Delta_{i_{1}j_{1}}u.(\overline{M}_{i_{2}j_{2}}(-\|q-a\|^{2}-\epsilon)^{n-2}+\ldots+$$
$$+\sum_{i_{1}j_{1}i_{2}j_{2}\ldots}\delta_{01\ldots(2n-1)}^{i_{1}j_{1}\ldots i_{n}j_{n}}\Delta_{i_{1}j_{1}}u.\overline{M}_{i_{2}j_{2}}\ldots\overline{M}_{i_{n}j_{n}}]\Omega_{2n}.$$
Note that in the right hand side above, except for the first two sums, all other sums vanish by simple computation, (for mor details see proof of proposition 4.1 in \cite{WW}).\\
$u$ is a locally bounded psh function on $\Omega,$ so there exists $C>0$ such that $\|\Delta_{ij}u\|_{L^{\infty}(\Omega)}\leq C$ for all $i,j,$ and there exists $C'>0$ such that $\|\Delta_{(2k)(2k+1)}u\|_{L^{\infty}(\Omega)}\leq C'$ for all $k.$
Then, by straightforward computation we get
$$\int_{B(a,s)}\sum_{k_{1}\ldots k_{n}}2^{n}\delta_{01\ldots(2n-1)}^{(2k_{1})(2k_{1}+1)\ldots(2k_{n})(2k_{n}+1)}\Delta_{(2k_{1})(2k_{1}+1)}u.(-\|q-a\|^{2}-\epsilon)^{n-1}dV\leq 2^{n}n!C'\int_{B(a,s)}(-\|q-a\|^{2}-\epsilon)^{n-1}dV,$$
and for $C>0$ large enough, we have
$$\int_{B(a,s)}\sum_{i_{1}j_{1}i_{2}j_{2}k_{3}\ldots k_{n}}2^{n-2}\delta_{01\ldots(2n-1)}^{i_{1}j_{1} i_{2}j_{2}(2k_{3})(2k_{3}+1)\ldots(2k_{n})(2k_{n}+1)}\Delta_{i_{1}j_{1}}u.\overline{M}_{i_{2}j_{2}}(-\|q-a\|^{2}-\epsilon)^{n-2}dV$$
$$\leq\int_{B(a,s)}\sum_{k_{1}\ldots k_{n}}2^{n}C\delta_{01\ldots (2n-1)}^{(2k_{1})(2k_{1}+1)\ldots}\overline{M}_{(2k_{1})(2k_{1}+1)}(-\|q-a\|^{2}-\epsilon)^{n-2}dV=2^{n}n!C\int_{B(a,s)}\|q-a\|^{2}(-\|q-a\|^{2}-\epsilon)^{n-2}dV$$
by the fact that $\|q-a\|^{2}=\sum_{k=0}^{n}M_{(2k)(2k+1)}.$
So by simple computation, we get
$$ \int_{B(a,s)}\Delta u\wedge(\Delta(\frac{-1}{\|q-a\|^{2}+\epsilon}))^{n-1}dV\leq C'\int_{B(a,s)}\frac{8^{n}n!}{(\|q-a\|^{2}+\epsilon)^{2n-2}}dV.$$
Then\begin{eqnarray*}
      \int_{\|q-a\|<s}(\Delta u\wedge(\Delta(\frac{-1}{\|q-a\|^{2}}))^{n-1})dV &=& \lim_{\epsilon\longrightarrow0}\int_{\|q-a\|<s}(\Delta u\wedge(\Delta(\frac{-1}{\|q-a\|^{2}+\epsilon}))^{n-1})dV \\
        &\leq& \lim_{\epsilon\longrightarrow0} S_{4n}C'\int_{0}^{s}\frac{8^{n}n!t^{4n-1}}{(\epsilon+t^{2})^{2n-2}}dt \\
        &=& S_{4n}C'\int_{0}^{s}\frac{8^{n}n!t^{4n-1}}{t^{4n-4}}dt\\
        &=& S_{4n}C'8^{n}n!\int_{0}^{s}t^{3}dt=S_{4n}C'\frac{8^{n}n!}{4}s^{4}
    \end{eqnarray*}
    So $$\int_{\{a\}}\Delta u\wedge(\Delta(\frac{-1}{\|q-a\|^{2}}))^{n-1}\leq\lim_{s\longrightarrow0}S_{4n}C'\frac{8^{n}n!}{4}s^{4}=0,$$
then $$\int_{\{a\}}\Delta u\wedge(\Delta(\frac{-1}{\|q-a\|^{2}}))^{n-1}=0.$$
On the other hand, by proposition 4.2 in \cite{WW}, we have for $0<s<r,$
$$ \int_{B(a,r)\setminus \overline{B}(a,s)}\Delta u\wedge(\Delta(\frac{-1}{\|q-a\|^{2}+\epsilon}))^{n-1}=\frac{1}{r^{4n-4}}\int_{B(a,r)}\Delta u\wedge\beta_{n}^{n-1}-\frac{1}{s^{4n-4}}\int_{B(a,s)}\Delta u\wedge\beta_{n}^{n-1}$$
tend $s$ to $0,$ we get
$$ \int_{B(a,r)\setminus\{a\}}\Delta u\wedge(\Delta(\frac{-1}{\|q-a\|^{2}+\epsilon}))^{n-1}=\frac{1}{r^{4n-4}}\int_{B(a,r)}\Delta u\wedge\beta_{n}^{n-1}-\nu_{u}(a),$$ where $\nu_{u}(a)$ is the Lelong number of $u$ at point $a.$
Since $u$ is bounded function, $\nu_{u}(a)=0.$ So by the first part of this proof, we have
\begin{eqnarray*}
  \int_{B(a,r)}\Delta u\wedge(\Delta(\frac{-1}{\|q-a\|^{2}}))^{n-1} &=& \int_{B(a,r)\setminus\{a\}}\Delta u\wedge(\Delta(\frac{-1}{\|q-a\|^{2}}))^{n-1}+\int_{\{a\}}\Delta u\wedge(\Delta(\frac{-1}{\|q-a\|^{2}}))^{n-1} \\
    &=&\frac{1}{r^{4n-4}}\int_{B(a,r)}\Delta u\wedge\beta_{n}^{n-1}.
\end{eqnarray*}
\end{proof}
\begin{lem}\label{6}
\begin{enumerate}
  \item Assume that $\nabla u\in L^{2}(\Omega).$ Then for $\delta>0$ small enough, we have $$ \int_{\Omega_{\delta}}|u_{\delta}(q)-u(q)|^{2}dV_{4n}(q)\leq C_{n}\|\nabla u\|^{2}_{L^{2}(\Omega)}\delta^{2},$$
  \item Assume that $\|\Delta u\|_{\Omega}<+\infty.$ Then for $\delta>0$ small enough, we have $$ \int_{\Omega_{\delta}}[\widehat{u}_{\delta}(q)-u(q)]dV_{4n}(q)\leq C_{n}\|\Delta_{\mathbb{H}^{n}} u\|_{\Omega}\delta^{2},$$
\end{enumerate}
where $\Delta u\wedge\beta_{n}^{n-1}=\Delta_{\mathbb{H}^{n}}u\Omega_{2n},$ and  $C_{n}>0$ is a constant depends only on $n.$
\end{lem}
\begin{proof}
For 1) see the last part in the proof of Theorem 3.1 in \cite{GKZ}.\\
2) It follows from Lelong-Jensen type formula (Theorem 5.1 in \cite{WW}) and lemma \ref{5.1}, that for $q\in\Omega_{\delta},$ $0<r<\delta,$ $r'=\frac{-1}{r^{2}},$ $\varphi(\xi)=\frac{-A}{\|\xi-q\|^{2}}$ where $A=(\frac{(2n)!}{8^{n}n!\pi^{2n}})^{\frac{1}{n}},$ and $B_{\varphi}(r')=\{\xi\in\Omega, \varphi(\xi)\leq r'\}.$
$$ \frac{1}{\sigma_{4n-1}}\int_{|\xi|=1}u(q+r\xi)dS_{4n-1}=u(q)+\int_{-\infty}^{r'}t^{2n-2}\int_{\frac{-1}{\|\xi-q\|^{2}}\leq t}\Delta u\wedge\beta^{n-1}_{n}dt$$
Using polar coordinates we get, for $q\in\Omega_{\delta}$
$$ \widehat{u}(q)-u(q)= \frac{1}{\sigma_{4n-1}\delta^{4n}}\int_{0}^{\delta}r^{4n-1}dr\int_{0}^{(\frac{-1}{r'})^{\frac{1}{2}}}s^{1-4n}(\int_{\|\xi-q\|\leq s}\Delta u\wedge\beta_{n}^{n-1})ds$$
So, by Fubini's theorem we have
\begin{eqnarray*}
  \int_{\Omega_{\delta}}(\widehat{u}-u)dV &\leq& a_{n}\delta^{-4n}\int_{0}^{\delta}r^{4n-1}dr\int_{0}^{r}s^{1-4n}(\int_{\|\xi-q\|\leq s}(\int_{\Omega}\Delta_{\mathbb{H}^{n}}u))ds \\
    &\leq& C_{n}\delta^{2}\|\Delta_{\mathbb{H}^{n}}u\|.
\end{eqnarray*}
\end{proof}
For giving us the H\"{o}lder norm estimate in $\overline{\Omega}$ of the solution $u,$ we need to apply the stability estimate with $u_{2}:=u_{\delta}.$ And in order to do that, we have to extend $u_{\delta}$ to $\Omega,$ since it is only defined on $\Omega_{\delta}.$
\begin{pro}\label{7}
Let $u\in PSH(\Omega)\cap L^{\infty}(\Omega)$ such that $u=\psi\in Lip_{2\beta}(\partial\Omega)$ on $\partial\Omega.$ Then there exist a constant $c_{0}=c_{0}(u)>0$ and $\delta_{0}$ small enough such that for any $0<\delta<\delta_{0}$ the function
$$\widetilde{u}_{\delta}=\left\{\begin{array}{ll}
                       \max\{u_{\delta},u+c_{0}\delta^{\beta}\} & \hbox{in $\Omega_{\delta}$;} \\
                       u+c_{0}\delta^{\beta}, &\hbox{in $\Omega\backslash\Omega_{\delta}$.}
                     \end{array}
                   \right.$$
is a bounded plurisubharmonic function on $\Omega$ and $(\widetilde{u}_{\delta})$ decreases to $u$ as $\delta$ decrease to $0.$
\end{pro}
For the proof we need the following result.
\begin{lem}\label{7.5}
Fix $\psi\in Lip_{2\alpha}(\partial\Omega),$ $f\in L^{p}(\Omega),$ $p>2$ and set $u:=u(\Omega,\psi,f).$ Then there exist $\varphi,\phi\in PSH(\Omega)\cap C^{\alpha}(\overline{\Omega})$ such that
\begin{enumerate}
  \item $\varphi(\xi)=\psi(\xi)=-\phi(\xi),$ $\forall\xi\in\partial\Omega.$
  \item $\varphi(q)\leq u(q)\leq-\phi(q)$ $\forall q\in\Omega.$
\end{enumerate}
\end{lem}
\begin{proof}
We are going to construct a weak barrier $b_{f}\in PSH(\Omega)\cap Lip_{1}(\Omega)$ for the Dirichlet problem $MA(\Omega,0,f)$ such that
\begin{itemize}
  \item $b_{f}(\xi)=0$ $\forall\xi\in\partial\Omega$
  \item $b_{f}\leq u(\Omega,0,f)$ in $\Omega$
  \item $|b_{f}(q)-b_{f}(\zeta)|\leq C|q-\zeta|$ $\forall q\in\Omega$ $\forall\zeta\in\Omega$
\end{itemize}
for some uniform constant $C>0.$
First, assume $f$ is bounded near $\partial\Omega,$ so $\exists K\subset\Omega$ $0\leq f\leq M$ on $\Omega\backslash K,$ where $K$ is a compact subset in $\Omega.$ \\
Set $b_{f}:=A\rho,$ where $\rho$ be a $C^{2}$ strictly psh defining function for $\Omega,$ by taking $A>0$ large enough so that $$ (\Delta b_{f})^{n}\geq MdV\geq fdV \ \ on \ \ \Omega\backslash K \ \ and \ \  b_{f}\leq m\leq u(\Omega,0,f) \ \ near \ \ K$$ where $m:=\min_{\Omega}u(\Omega,0,f).$
Then $ (\Delta b_{f})^{n}\geq(\Delta u(\Omega,0,f))^{n}$ on $\Omega\backslash K,$ and $b_{f}\leq u(\Omega,0,f)$ on $\partial(\Omega\backslash K).$  This implies, by the comparaison principle ( Corollary 1.1 in \cite{WZ}) that
$b_{f}\leq u(\Omega,0,f)$ in $\Omega.$ \\
For the general case, $f$ is not bounded near $\partial\Omega.$ Fix a large ball $\mathbb{B}\subset\mathbb{H}^{n}$ so that $\Omega\Subset\mathbb{B}\subset\mathbb{H}^{n}.$ \\ Set $\widetilde{f}:=f$ in $\Omega$ and $\widetilde{f}=0$ in $\mathbb{B}\backslash\Omega.$  By the first part of this proof, we can find a barrier function \\ $b_{\widetilde{f}}\in PSH(\mathbb{B})\cap C^{2}(\mathbb{B})$ for the Dirichlet problem $MA(\mathbb{B},0,\widetilde{f}).$ Set $h:=u(\Omega,-b_{\widetilde{f}},0).$ \\
Since $-b_{\widetilde{f}}\in C^{2}(\partial\Omega),$ by Corollary \ref{cor2} $h$ is Lipshitz on $\Omega.$ Set $b_{f}:=h+b_{\widetilde{f}}\in PSH(\Omega)\cap Lip_{1}(\Omega)$ is a barrier function for $MA(\Omega,0,f).$ Moreover, by corollary \ref{cor2} we have $u(\Omega,\pm\psi,0)$ is H\"{o}lder continuous of order $\alpha,$ where $\psi\in C^{2\alpha}(\partial\Omega).$ Then, the functions $\varphi:=u(\Omega,\psi,0)+b_{f}$ and $\phi:=u(\Omega,-\psi,0)+b_{f}$ belong to $PSH(\Omega)\cap Lip_{\alpha}(\overline{\Omega})$ and satisfies 1) and 2).
\end{proof}
\begin{proof} of Proposition \ref{7}\\
Using Lemma \ref{7.5}, and follow the same proof of Proposition 2.1 in \cite{GKZ}.
\end{proof}

Now, we are in position to prove our main tool, which is the following.
\begin{thm}\label{8}
Let $\Omega$ be a bounded pseudoconvex domain of $\mathbb{H}^{n}.$ Assume that $\psi\in Lip_{2\beta}(\partial\Omega)$ and fix $f\in L^{p}(\Omega)$  for some $p>2.$ Let $u$ be the unique solution to \ref{DP12} for $d\nu=fdV.$
\begin{enumerate}
  \item If $\nabla u$ belongs to $L^{2}(\Omega),$ then $u\in Lip_{\beta'}(\overline{\Omega})$ for all $\beta'<\min(\beta,\gamma_{2}).$
  \item If the total mass of $\Delta_{\mathbb{H}^{n}} u$ is finite, then $u\in Lip_{\beta''}(\overline{\Omega})$ for all $\beta''<\min(\beta,2\gamma_{1}).$
\end{enumerate}
where $\frac{1}{p}+\frac{1}{q}=1,$ and $\gamma_{1},\gamma_{2}$ are defined in Theorem \ref{4}.
\end{thm}
\begin{proof}
1)We have $f\in L^{p}(\Omega),$ $p>2.$ By \cite{SM}, we have the solution $u\in PSH(\Omega)\cap C(\Omega)$ is a continuous plurisubharmonic function. Then, we have to show that $u$ is H\"{o}lder continuous on $\overline{\Omega}.$ Given $0<\gamma<\frac{2}{qn+2+\frac{nq}{\frac{2}{q}-1}}.$ Applying the stability estimate Theoreme \ref{4} with $r=2,$ $u_{2}=\widetilde{u}_{\delta}$ and $u_{1}:=u+c_{0}\delta^{\beta}$ we get
$$ \sup_{\Omega}[\widetilde{u}_{\delta}-(u+c_{0}\delta^{\beta})]\leq C\|(\widetilde{u}_{\delta}-u-c_{0}\delta^{\beta})_{+}\|^{\gamma}_{L^{2}(\Omega)}.$$
Since $\widetilde{u}_{\delta}=u+c_{0}\delta^{\beta}$ in $\Omega\backslash \Omega_{\delta},$ we have
$$ \sup_{\Omega_{\delta}}[u_{\delta}-u-c_{0}\delta^{\beta}]\leq C\|(u_{\delta}-u-c_{0}\delta^{\beta})_{+}\|^{\gamma}_{L^{2}(\Omega_{\delta})}.$$
Since $((u_{\delta}-u-c_{0}\delta^{\beta})_{+}\leq u_{\delta}-u$ and by Lemma \ref{6}, we have
$$\sup_{\Omega_{\delta}}(u_{\delta}-u)\leq c_{0}\delta^{\beta}+ C\|(u_{\delta}-u\|^{\gamma}_{L^{2}(\Omega_{\delta})}\leq c_{0}\delta^{\beta}+CC_{n}^{\frac{\gamma}{2}}\|\nabla u\|^{\gamma}_{L^{2}(\Omega)}\delta^{\gamma}.$$
Then, $$ \sup_{\Omega_{\delta}}(u_{\delta}-u)\leq A\delta^{\min\{\beta,\gamma\}},$$ for $\delta$ small enough, where $A=c_{0}+CC_{n}^{\frac{\gamma}{2}}\|\nabla u\|^{\gamma}_{L^{2}(\Omega)}.$ This proves the first part of this result.\\
2) Given $0<\gamma<\gamma_{1}.$
$$ u_{\delta}(q)\leq u(q)+c_{0}\delta^{\beta}\Longrightarrow \widehat{u}_{\delta}\leq u_{\delta}<u+c_{0}\delta^{\beta} \ \ on \ \ \partial\Omega_{\delta}.$$
The function $$u'_{\delta}=\left\{\begin{array}{ll}
                       \max\{\widehat{u}_{\delta},u+c_{0}\delta^{\beta}\} & \hbox{in $\Omega_{\delta}$;} \\
                       u+c_{0}\delta^{\beta}, &\hbox{in $\Omega\backslash\Omega_{\delta}$.}
                     \end{array}
                   \right.$$
is a bounded plurisubharmonic function on $\Omega,$ continuous in $\overline{\Omega}.$ Using Theorem \ref{4} with \\ $u_{1}:=u+c_{0}\delta^{\beta},$ $u_{2}:=u'_{\delta}$ and $r=1,$ we get
$$ \sup_{\Omega}[u'_{\delta}-u-c_{0}\delta^{\beta}]\leq C\|(u'_{\delta}-u-c_{0}\delta^{\beta})_{+}\|^{\gamma}_{L^{1}(\Omega)}.$$
We have $u'_{\delta}=u+c_{0}\delta^{\beta}$ in $\Omega\backslash\Omega_{\delta},$ hence
$$ \sup_{\Omega_{\delta}}[\widehat{u}_{\delta}-u-c_{0}\delta^{\beta}]\leq C\|(\widehat{u}_{\delta}-u-c_{0}\delta^{\beta})_{+}\|^{\gamma}_{L^{1}(\Omega_{\delta})}.$$
Also, we have $(\widehat{u}_{\delta}-u-c_{0}\delta^{\beta})_{+}\leq\widehat{u}_{\delta}-u,$ so we get
\begin{eqnarray*}
  \sup_{\Omega_{\delta}}(\widehat{u}_{\delta}-u) &\leq& c_{0}\delta^{\beta}+C\|\widehat{u}_{\delta}-u\|^{\gamma}_{L^{1}(\Omega_{\delta})}\\
    &\leq& c_{0}\delta^{\beta}+CC^{\gamma}_{n}\|\Delta_{\mathbb{H}^{n}} u\|^{\gamma}_{\Omega}\delta^{2\gamma}
\end{eqnarray*}
Then, $\sup_{\Omega_{\delta}}(\widehat{u}_{\delta}-u)\leq M_{1}\delta^{\min\{\beta,2\gamma\}},$ for $\delta$ small enough, and $M_{1}=c_{0}+CC_{n}^{\alpha}\|\Delta_{\mathbb{H}^{n}} u\|_{\Omega}^{\gamma}.$\\ By Lemma \ref{5}, we have
$$\sup_{\Omega_{\delta}}(u_{\delta}-u)\leq M_{2}\delta^{\min\{\beta,2\gamma\}},$$
for $\delta$ small enough, and some uniform constant $M_{2}>0.$ This finishes the last part of this result.
\end{proof}
Now, we are in the last part of this paper. We are going to prove the main Theorem, using these following results.
\begin{lem}\label{9}
Let $u,$ $v$ be continuous functions on $\overline{\Omega}$ and be plurisubharmonic functions in $\Omega,$ such that $u\geq v$ in $\Omega$ and $u=v$ on $\partial\Omega.$ Then
$$ \int_{\Omega}\Delta u\wedge\beta_{n}^{n-1}\leq\int_{\Omega}\Delta v\wedge\beta_{n}^{n-1}$$
$$\int_{\Omega}d_{0}u\wedge d_{1}u\wedge\beta_{n}^{n-1}\leq 2\int_{\Omega}\gamma(v.u)\wedge\beta_{n}^{n-1}+\int_{\Omega}d_{0}v\wedge d_{1}v\wedge\beta_{n}^{n-1},$$
where $\gamma(u,v):=\frac{1}{2}(d_{0}u\wedge d_{1}v-d_{1}u\wedge d_{0}v)$, $\beta_{n}:=\frac{1}{8}\Delta(\|q\|^{2}).$ Furthermore, if
$\int_{\Omega}d_{0}v\wedge d_{1}v\wedge\beta_{n}^{n-1}<+\infty$ then $\int_{\Omega}d_{0}u\wedge d_{1}u\wedge\beta_{n}^{n-1}<+\infty.$
\end{lem}
\begin{proof}
First, we set $u_{\epsilon}=\max\{u-\epsilon,v\}$ for $\epsilon>0.$ We have $u,$ $v$ are continuous and $u=v$ on $\partial\Omega,$ so $u_{\epsilon}=v$ in a neighborhood of $\partial\Omega.$ Let $\{\Omega_{j}\}$ a hyperconvex open in $\Omega$ such that $\{u_{\epsilon}\neq v\}\subset\subset\Omega_{1}\subset\subset\ldots\Omega_{j}\subset\subset\ldots\Omega$ and $\{\chi_{j}\}\subset C^{\infty}_{0}(\Omega)$ such that $\chi_{j}\equiv1$ in a neighborhood of $\overline{\Omega}_{j}$ and $\chi_{j}\nearrow 1.$ By Lemma \ref{200}, we have
\begin{eqnarray*}
  \int_{\Omega}\chi_{j}\Delta u_{\epsilon}\wedge\beta_{n}^{n-1} &=& -\int_{\Omega}d_{0}\chi_{j}\wedge d_{1}u_{\epsilon}\wedge\beta_{n}^{n-1} \\
    &=& -\int_{\Omega\setminus\overline{\Omega}_{j}}d_{0}\chi_{j}\wedge d_{1}u_{\epsilon}\wedge\beta_{n}^{n-1} \\
    &=& -\int_{\Omega\setminus\overline{\Omega}_{j}}d_{0}\chi_{j}\wedge d_{1}v\wedge\beta_{n}^{n-1}\\
    &=& \int_{\Omega}\chi_{j}\Delta v\wedge\beta_{n}^{n-1}
\end{eqnarray*}
Letting $j$ tend to $+\infty,$ we get
$$ \int_{\Omega}\Delta u_{\epsilon}\wedge\beta_{n}^{n-1}= \int_{\Omega}\Delta v\wedge\beta_{n}^{n-1}.$$
Since $u\geq v$ in $\Omega,$ we have $u_{\epsilon}\nearrow u$ in $\Omega.$ By the monotone convergence theorem, we get
$$\Delta u_{\epsilon}\wedge\beta_{n}^{n-1}\longrightarrow\Delta u\wedge\beta_{n}^{n-1}.$$ So
\begin{eqnarray*}
  \int_{\Omega}\Delta u\wedge\beta_{n}^{n-1} &\leq& \liminf_{\epsilon\longrightarrow 0}\int_{\Omega}\Delta u_{\epsilon}\wedge\beta_{n}^{n-1} \\
    &=& \int_{\Omega}\Delta v\wedge\beta_{n}^{n-1}.
\end{eqnarray*}
The first one follows.
For the second one, we have also $u,v$ are continuous and $u=v$ on $\partial\Omega,$ so we set $u_{\epsilon}:=\max\{u-\epsilon,v\}=v$ in a neighborhood of $\partial\Omega$ and $u_{\epsilon}\geq v$ on $\Omega.$
We have
\begin{eqnarray*}
             \int_{\Omega}d_{0}v\wedge d_{1}v\wedge\beta_{n}^{n-1}-\int_{\Omega}d_{0}u_{\epsilon}\wedge d_{1}u_{\epsilon}\wedge\beta_{n}^{n-1}+2\int_{\Omega}\gamma(v.u_{\epsilon})\wedge\beta_{n}^{n-1} &=& \int_{\Omega}d_{0}(v-u_{\epsilon})\wedge d_{1}(v+u_{\epsilon})\wedge\beta_{n}^{n-1} \\
              &=& \int_{\Omega}(u_{\epsilon}-v)\wedge \Delta(v+u_{\epsilon})\wedge\beta_{n}^{n-1}\geq0.
           \end{eqnarray*}
Then, $$\int_{\Omega}d_{0}v\wedge d_{1}v\wedge\beta^{n-1}\geq\int_{\Omega}d_{0}u_{\epsilon}\wedge d_{1}u_{\epsilon}\wedge\beta_{n}^{n-1}-2\int_{\Omega}\gamma(v.u_{\epsilon})\wedge\beta_{n}^{n-1}$$
By convergence Theorem, we have $$\int_{\Omega}d_{0}u_{\epsilon}\wedge d_{1}u_{\epsilon}\wedge\beta_{n}^{n-1}-2\int_{\Omega}\gamma(v.u_{\epsilon})\wedge\beta_{n}^{n-1}\longrightarrow\int_{\Omega}d_{0}u\wedge d_{1}u\wedge\beta_{n}^{n-1}-2\int_{\Omega}\gamma(v.u)\wedge\beta_{n}^{n-1} \ \ as \ \ \epsilon\searrow0.$$ Thus $$\int_{\Omega}d_{0}v\wedge d_{1}v\wedge\beta_{n}^{n-1}+2\int_{\Omega}\gamma(v.u)\wedge\beta_{n}^{n-1}\geq\int_{\Omega}d_{0}u\wedge d_{1}u\wedge\beta_{n}^{n-1}.$$
By Corollary 3.1 in \cite{WZ}, we have
\begin{eqnarray*}
  \int_{\Omega}d_{0}u\wedge d_{1}u\wedge\beta_{n}^{n-1} &\leq& \int_{\Omega}d_{0}b_{\gamma}\wedge d_{1}b_{\gamma}\wedge\beta_{n}^{n-1}+2|\int_{\Omega}\gamma(b_{\gamma},u)\wedge\beta_{n}^{n-1}| \\
    &\leq& \int_{\Omega}d_{0}b_{\gamma}\wedge d_{1}b_{\gamma}\wedge\beta_{n}^{n-1}+2\sqrt{\int_{\Omega}d_{0}u\wedge d_{1}u\wedge\beta_{n}^{n-1}}\sqrt{\int_{\Omega}d_{0}b_{\gamma}\wedge d_{1}b_{\gamma}\wedge\beta_{n}^{n-1}}
\end{eqnarray*}
and we obtain
$$\sqrt{\int_{\Omega}d_{0}u\wedge d_{1}u\wedge\beta_{n}^{n-1}}\leq 2\sqrt{\int_{\Omega}d_{0}b_{\gamma}\wedge d_{1}b_{\gamma}\wedge\beta_{n}^{n-1}}+\frac{\int_{\Omega}d_{0}b_{\gamma}\wedge d_{1}b_{\gamma}\wedge\beta_{n}^{n-1}}{\sqrt{\int_{\Omega}d_{0}u\wedge d_{1}u\wedge\beta_{n}^{n-1}}}.$$
So necessary we have $\int_{\Omega}d_{0}u\wedge d_{1}u\wedge\beta_{n}^{n-1}<+\infty.$
This finishes the lemma.
\end{proof}
\begin{pro}\label{10}
Fix $0\leq f\in L^{p}(\Omega)$ $(p>2).$ If $\psi\in C^{1,1}(\partial\Omega),$ Then $\Delta_{\mathbb{H}^{n}} u(\Omega,\psi,0)$ has finite mass in $\Omega.$ Moreover $\Delta_{\mathbb{H}^{n}} u(\Omega,\psi,f)$ also has finite mass in $\Omega.$
\end{pro}
\begin{proof}
We fix a defining function $\rho$ of $\Omega.$ Setting $\Omega=\{\rho<0\},$ $\rho\in C^{2}(\overline{\Omega}).$\\
First, we claim that
$$ h(q)=\sup\{v(q): v\in PSH(\Omega)\cap C(\overline{\Omega}), v\leq\psi \ \ on \ \ \partial\Omega\}$$ is psh function in $\Omega$ and is Lipschitz continuous in $\overline{\Omega},$ it satisfies $h=\psi$ on $\partial\Omega.$ Moreover
$$\int_{\Omega}\Delta h\wedge\beta_{n}^{n-1}<+\infty.$$
Assume $f=0$, set $u:=u(\Omega,\psi,0),$  we may choose $A>0$ big enough such that $A\rho+h\leq u$ in a neighborhood of $F\Subset\Omega,$ as $\rho<-\epsilon$ in $F$ for some $\epsilon>0,$ and $(\Delta(A\rho+h))^{n}\geq (\Delta A\rho)^{n}\geq 0$ in $\Omega\backslash F,$ by comparaison principle (Corollary 1.1 in \cite{WZ}), we have $A\rho+h\leq u$ in $\Omega\backslash F.$ \\ Therefore, $b:=A\rho+h\leq u$ in $\Omega$ and $b$ is Lipschitz continuous in $\overline{\Omega}.$ By Lemma \ref{9} and the fact that $\rho$ is $C^{2}$ smooth in a neighborhood of $\overline{\Omega},$ by using the claim above we get
$$ \int_{\Omega}\Delta u\wedge\beta_{n}^{n-1}\leq\int_{\Omega}\Delta b\wedge\beta_{n}^{n-1}<+\infty.$$
For the last part of this proposition, we are going to prove that $\Delta_{\mathbb{H}^{n}} u(\Omega,0,f)$ has finite mass in $\Omega.$ \\ Let $\widetilde{f}$ be the trivial extension of $f$ to a large ball $B$ containing $\Omega.$ Let $b_{\widetilde{f}}\in C^{2}(\mathbb{B})$ be a psh barrier for $MA(\mathbb{B},0,\widetilde{f})$ (see the proof of Lemma \ref{7.5} ). Then $b_{f}:=u(\Omega,-b_{\widetilde{f}},0)+b_{\widetilde{f}}$ is psh barrier for $MA(\Omega,0,f).$ Since $b_{\widetilde{f}}$ is smooth, we have $\Delta_{\mathbb{H}^{n}} b_{f}$ has finite mass in $\Omega.$ \\
On the other hand, we have $(\Delta b_{f})^{n}\geq fdV$ in $\Omega,$ so $b_{f}\leq u(\Omega,0,f)$ in $\Omega$ by comparaison principle (Corollary 1.1 in \cite{WZ}). Using Lemma \ref{9}, we get
$$\int_{\Omega}\Delta u(\Omega,0,f)\wedge\beta_{n}^{n-1}\leq\int_{\Omega}\Delta b_{f}\wedge\beta_{n}^{n-1}<+\infty.$$
Now set $v:=u(\Omega,0,f)+u(\Omega,\psi,0),$ $v$ is a psh function in $\Omega$ such that $v=\psi$ on $\partial\Omega$ and
$(\Delta v)^{n}\geq fdV$ in $\Omega.$ Since $\psi$ is $C_{1.1}$ in $\partial\Omega,$ we have $\Delta_{\mathbb{H}^{n}} u(\Omega,\psi,0)$ has finite mass in $\Omega,$ and $\Delta_{\mathbb{H}^{n}} u(\Omega,0,f)$ is so. Then $\Delta_{\mathbb{H}^{n}} v$ has finite mass in $\Omega,$ with $v\leq u(\Omega,\psi,f)$ in $\Omega.$ So by Lemma \ref{9}, we get
$$ \int_{\Omega}\Delta u(\Omega,\psi,f)\wedge\beta_{n}^{n-1}\leq\int_{\Omega}\Delta v\wedge\beta_{n}^{n-1}<+\infty.$$
For the proof of the claim, we let the reader to see the proof of Lemma 3.5 in \cite{N}.
\end{proof}
\begin{pro}\label{11}
Fix $0\leq f\in L^{p}(\Omega)$ $(p>2).$ If $\psi\in C^{1,1}(\partial\Omega),$ Then $\nabla u(\Omega,\psi,f)\in L^{2}(\Omega).$
\end{pro}
\begin{proof}
First, we claim that : for $0\leq\gamma<\frac{1}{2},$ the function $\rho_{\gamma}=-|\rho|^{1-\gamma},$ $\rho$ as in the proof of proposition \ref{10}, setting $(\Delta\rho)^{n}\geq g\beta^{n}$ on $\overline{\Omega},$ with $g>0$ $\rho_{\gamma}\in PSH(\Omega)\cap Lip_{1-\gamma}(\overline{\Omega})$ and satisfies
$$ \int_{\Omega}d_{0}\rho_{\gamma}\wedge d_{1}\rho_{\gamma}\wedge\beta_{n}^{n-1}<+\infty.$$
Now, assume that $f(q)\leq C|\rho(q)|^{-n\gamma}$ near $\partial\Omega$ for some $C>0.$ So there is a compact subset $E\Subset\Omega$ such that $f(q)\leq C|\rho(q)|^{-n\gamma}$ in $\Omega\backslash E.$\\
Then, we have \begin{eqnarray*}
                (\Delta\rho_{\gamma})^{n}=(d_{0}d_{1}(-(-\rho)^{1-\gamma})^{n} &=& ((1-\gamma)|\rho|^{-\gamma}\Delta\rho+\gamma(1-\gamma)|\rho|^{-1-\gamma}d_{0}\rho\wedge d_{1}\rho)^{n} \\
                  &\geq& (1-\gamma)^{n}|\rho|^{-n\gamma}g\beta_{n}^{n} \\
                   &\geq& \frac{g(1-\gamma)^{n}}{C}f\beta_{n}^{n} \ \ in \ \ \Omega\backslash E.
              \end{eqnarray*}
Therefore, we may choose $A>0$ big enough such that $b_{\gamma}:=A\rho_{\gamma}+h\leq u$ in a neighborhood of $E,$ and $$ (\Delta b_{\gamma})^{n}\geq (\Delta A\rho_{\gamma})^{n}\geq f\beta_{n}^{n} \ \ in \ \ \Omega\backslash E,$$
where $h$ as in the proof of proposition \ref{10}.
Then, by the comparison principle (Corollary 1.1 in \cite{WZ}), we obtain $b_{\gamma}\leq u$ in $\Omega\backslash E.$
So $b_{\gamma}\leq u$ in $\Omega$ and $b_{\gamma}\in Lip_{1-\gamma}(\overline{\Omega}).$ By Lemma \ref{9}, we have $$\int_{\Omega}d_{0}u\wedge d_{1}u\wedge\beta_{n}^{n-1}\leq\int_{\Omega}d_{0}b_{\gamma}\wedge d_{1}b_{\gamma}\wedge\beta_{n}^{n-1}+2\int_{\Omega}\gamma(b_{\gamma},u)\wedge\beta_{n}^{n-1},$$
and by the claim above, we get $\int_{\Omega}d_{0}u\wedge d_{1}u\wedge\beta_{n}^{n-1}<+\infty.$\\
For the general case, we set $f=0,$ we obtain $\int_{\Omega}d_{0}u\wedge d_{1}u\wedge\beta_{n}^{n-1}<+\infty,$ by the first part of this proof.
Now, for $f\neq0.$ Set $v:=u(\Omega,\psi,0)+b_{f},$ where $b_{f}$ is the plurisubharmonic barrier constructed in the proof of Lemma \ref{7.5}. We have $v=\psi+0=u$ on $\partial\Omega,$ $(\Delta v)^{n}\geq (\Delta b_{f})^{n}\geq fdV$ in $\Omega,$ so $v\leq u$ in $\Omega.$ Moreover, we have $\nabla u(\Omega,\psi,0)\in L^{2}(\Omega)$ and $\nabla b_{f}\in L^{2}(\Omega)$ hence $\nabla v\in L^{2}(\Omega).$
By Lemma \ref{9}, we get $$\int_{\Omega}d_{0}u\wedge d_{1}u\wedge\beta_{n}^{n-1}<+\infty.$$
Now, we prove the claim. we have
 $$\Delta\rho_{\gamma}=d_{0}d_{1}(-(-\rho)^{1-\gamma})= (1-\gamma)|\rho|^{-\gamma}\Delta\rho+\gamma(1-\gamma)|\rho|^{-1-\gamma}d_{0}\rho\wedge d_{1}\rho,$$
 and $$ d_{0}\rho_{\gamma}\wedge d_{1}\rho_{\gamma}\wedge\beta_{n}^{n-1}=(1-\gamma)^{2}|\rho|^{-2\gamma}d_{0}\rho\wedge d_{1}\rho\wedge\beta^{n-1}$$
 Since $-2\gamma>-1,$ we have $\int_{\Omega}d_{0}\rho_{\gamma}\wedge d_{1}\rho_{\gamma}\wedge\beta_{n}^{n-1}<+\infty.$ Then, the proof is finished.
\end{proof}
\textbf{Proof of main Theorem}
According to Proposition \ref{10} and Proposition \ref{11}, the assumptions of Theorem \ref{8} are satisfied. Thus, the main Theorem follows.

\par \textbf{ Ibn Tofail University, Faculty of Sciences, Kenitra, Morocco.}
\par \textbf{ E-mail adress: fadoua.boukhari@uit.ac.ma }


\addcontentsline{toc}{chapter}{Bibliographie}
\begin{thebibliography}{09}
\bibitem[A]{A} S.Alesker, Non-commutative linear algebra and plurisubharmonic functions of quaternionic variables. Bull.Sci.Math,127,1-35 (also math.CV/0104209)(2003).
\bibitem[A1]{A1} S.Alesker, Valuations on convex sets, non-commutative determinants, and pluripotential theory, Adv. Math. 195(2)(2005)561-595.
\bibitem[A2]{A2} S.Alesker, Pluripotential theory on quaternionic manifolds, J. Geom. Phys. 62(5)(2012) 1189-1206.
\bibitem[A3]{A3} S.Alesker, Quaternionic Monge-Ampère Equations, The Journal of Geometric Analysis, Vol 13, Number 2, 2003
\bibitem[BT]{BT} E.Bedford, B. A. Taylor. A new capacity for plurisubharmonic functions. Acta Math. 149 (1-2) (1982) 1-40.
\bibitem[C]{C} M.Charabati, Modulus of continuity of solutions to complex Hessian equations. Internat. J. Math. 27(2016),no.1, 1650003,24pp.
\bibitem[DW]{DW} D. Wan, Cegrell’s classes and a variational approach for the quaternionic Monge-Amp`ere equation.  arXiv:1802.08411v1, 2018.
\bibitem[EGZ]{EGZ} P. Eyssidieux, V. Guedj and A. Zeriahi, ‘Singular K\"{a}hler–Einstein metrics’, Preprint, 2006.
\bibitem[GZ17]{GZ17} V. Guedj, A.Zeriahi: Degenerate Complex Monge-Amp`ere Equations. Tracts in mathematics 26. European Mathematical Society.
\bibitem[GKZ]{GKZ} V. Guedj, S. Kolodziej, A. Zeriahi, H\"{o}lder continuous solutions to Monge- Ampère equations. Bull.London Math. Soc.40(2008) 1070-1080.
\bibitem[H]{H} L. H\"{o}rmander. Notions of convexity, volume 127 of Progress in Mathematics. Birkh\"{a}user Boston, Inc., Boston, MA, 1994.
\bibitem[K]{K} M.Klimek, Pluripotentiel theory, Vol. 6 of London Mathematical Society Monographs. New Series, The Clarendon Press, Oxford University Press, New York, 1991, Oxford Science Publications.
\bibitem[Ko]{Ko} S. Kolodziej, The complex Monge-Amp\`{e}re equation and pluripotential theory, Memoirs Amer. Math. Soc. 178 (2005) 64p.
\bibitem[Ko1]{Ko1} S. Kolodziej, Equicontinuity of families of plurisubharmonic functions with bounds on their Monge-Amp\`{e}re masses, Math.Z.240,835-847(2002).
\bibitem[N]{N} N.C. Nguyen, H\"{o}lder continuous solutions to complex Hessian equations, preprint arXiv: 1301.0710v2.
\bibitem[SM]{SM} M. Sroka, Weak solutions to quaternionic Monge-Amp\`{e}re equation, arXiv:1807.02482v1,2018.
\bibitem[T]{T} M. Tsuji, Potential Theory in Modern Function Theory, Maruzen, Tokyo, 1959.
\bibitem[W]{W} W. Wang, On the optimal control method in quaternionic analysis. Bull. Sci. Math. 135(8):988-1010,2011.
\bibitem[W1]{W1}W. Wang, The k-Cauchy–Fueter complex, Penrose transformation and Hartogs phenomenon for quaternionic k-regular functions, J. Geom. Phys. 60 (2010), no. 3, 513–530.
\bibitem[W2]{W2}W. Wang, The continuity and range of the quaternionic Monge–Ampère operator on quaternionic space, Math. Z. 285 (2017), 461–478
\bibitem[Wa]{Wa} Walsh, J.B. Continuity of envelopes of plurisubharmonic functions. J. Math. Mech. 18, 143–148 (1968)
\bibitem[WK]{WK}D.Wan, Q.Kang, Potential theory in several quaternionic variables, Michigan Math. J. 66 (2017), 3–20.
\bibitem[WW]{WW} D.Wan, W.Wang, On quaternonic Monge Ampère operator, closed positive currents and Lelong-Jensen type formula on quaternionic space,.Bull. Sci. math. 141 (2017) 267–311.
\bibitem[WW1]{WW1} D.Wan, W.Wang, Viscosity Solutions to quaternionic Monge-Amp\`{e}re equations, Nonlinear Anal., 140:69–81, 2016.
\bibitem[WZ]{WZ} D.Wan, W.Zhang, Quasicontinuity and maximality of quaternionic plurisubharmonic functions, J. Math. Anal. Appl. 424 (2015) 86-103.
\bibitem[Z]{Z} J.Zhu, Dirichlet problem of quaternionic Monge-Ampère equations, Israel Journal of Mathematics 214 (2016), 597–619.

    \end{thebibliography}
\end{document}